\documentclass[12pt]{amsart}

\usepackage{amssymb}

\usepackage{color}
	\definecolor{mygreen}{rgb}{0,.4,0}
	\definecolor{myblue}{rgb}{0,0,.5}
	\definecolor{mymagenta}{cmyk}{0,.6,0,0}
	
	\definecolor{marcelo}{rgb}{.2,.6,.2}
	\definecolor{aaron}{rgb}{.8,.2,.2}

\usepackage[bookmarks=false,colorlinks,linkcolor=myblue,citecolor=mygreen]{hyperref}

\usepackage{subfigure}
\usepackage[width=.85\textwidth,font=small,labelfont=sc]{caption}

\usepackage[left=1.25in,right=1.25in,top=1.6in,bottom=1.5in]{geometry}
\date{27 March 2014}

\theoremstyle{plain}
\newtheorem{theorem}{Theorem}
\newtheorem{corollary}[theorem]{Corollary}
\newtheorem{lemma}[theorem]{Lemma}
\newtheorem{proposition}[theorem]{Proposition}

\theoremstyle{definition}
\newtheorem{example}[theorem]{Example} 

\theoremstyle{remark}
\newtheorem*{remark}{Remark} 

\newcommand{\demph}[1]{\textcolor{myblue}{\it #1}}

\def\qand{\quad\hbox{and}\quad}

\newcommand{\qqand}{\qquad\text{and}\qquad}

\def\id{\mathrm{id}}
\def\unit{\iota}
\def\counit{\varepsilon}
\def\field{\Bbbk}
\def\kk{\field}

\def\tr{\operatorname{trace}}

\renewcommand{\ker}{\mathrm{Ker}}
\newcommand{\im}{\mathrm{Im}}

\newcommand{\bC}{\mathbb{C}}

\newcommand{\gr}[1]{\mathop{\mathrm{gr}}#1} 
\newcommand{\End}[1]{\mathop{\mathrm{End}}#1} 
\newcommand{\Conv}[1]{\End(#1)} 

\newcommand{\tilH}{\widetilde H}
\newcommand{\frakS}{\mathfrak S}
\newcommand{\frakg}{\mathfrak g}

\newcommand{\contra}[1]{#1^{\vee}}
\newcommand{\calP}{\mathcal P} 
\newcommand{\calT}{\mathcal T}
\newcommand{\cT}{\contra{\calT}}
\newcommand{\calS}{\mathcal{S}}
\newcommand{\calU}{\mathcal{U}}
\newcommand{\Kc}{\mathcal{K}}
\newcommand{\cKc}{\contra{\Kc}}

\newcommand{\Map}{T}
\newcommand{\dMap}{\Map^\ast}
\newcommand{\bMap}{\Map^\uparrow}
\newcommand{\uMap}{\Map_\downarrow}

\newcommand{\apode}{\mathrm{S}}

\newcommand{\sym}{\textsl{Sym}}
\newcommand{\ssym}{\frakS\sym}

\newcommand{\qsym}{\textsl{QSym}}

\newcommand{\tH}{\mathbf{H}} 
\newcommand{\tP}{\mathbf{P}} 
\newcommand{\egfh}{\mathbf{h}} 
\newcommand{\egfp}{\mathbf{p}} 
\newcommand{\egfa}{\mathbf{a}}

\newcommand{\wE}{\mathbf{E}} 
\newcommand{\wL}{\mathbf{L}} 
\newcommand{\tPi}{\mathbf{\Pi}} 
\newcommand{\tSig}{\mathbf{\Sigma}} 

\newcommand{\Eul}[1]{\mathop{\mathsf{e}^{(#1)}}} 
\newcommand{\mul}[1]{\mathop{\mathsf{mul}}#1} 
\newcommand{\emul}[1]{\mathop{\mathsf{emul}}#1}
\newcommand{\omul}[1]{\mathop{\mathsf{omul}}#1}
\newcommand{\expmul}[1]{\mathop{\mathsf{xmul}}#1} 
\newcommand{\pal}[1]{\mathsf{pal}#1} 
\newcommand{\epal}[1]{\mathsf{epal}#1} 
\newcommand{\opal}[1]{\mathsf{opal}#1} 
\newcommand{\nopal}[1]{\mathsf{npal}#1} 

\newcommand{\inv}{\mathsf{inv}} 

\newcommand{\abs}[1]{\lvert#1\rvert} 

\newcommand{\length}[1]{\ell\left(#1\right)}
\newcommand{\rev}[1]{\widetilde{#1}}

\def\cmrg{{\textcolor{mymagenta}{|}}} 

\newcommand{\shuff}{\mathrel{\raise1pt\hbox{$\scriptscriptstyle\sqcup{\mskip-4mu}\sqcup$}}}

\author{Marcelo Aguiar}
\address[Aguiar]{
	Department of Mathematics\\
        Cornell University\\
        Ithaca, NY\, 14853 
        }
\email{maguiar@math.cornell.edu}
\urladdr{http://www.math.cornell.edu/{\small$\sim$}maguiar}

\author{Aaron Lauve}
\address[Lauve]{
	Department of Mathematics and Statistics \\
	Loyola University Chicago \\
	Chicago, IL\, 60660 
        }
\email{lauve@math.luc.edu}
\urladdr{http://www.math.luc.edu/{\small$\sim$}lauve}

\thanks{Aguiar supported in part by NSF grant DMS-1001935.
Lauve supported in part by NSA grant \#H98230-12-1-0286.
This work was started at Texas A\&M University and both authors are
thankful to this institution.}

\title[{A}dams operators on graded connected {H}opf algebras]{The characteristic polynomial of the {A}dams operators
on graded connected {H}opf algebras}


\begin{document}

\begin{abstract}
The Adams operators $\Psi_n$ on a Hopf algebra $H$ are the convolution powers of the identity of $H$. They are also called Hopf powers or Sweedler powers. We study the Adams operators when $H$ is graded connected. 
The main result is a complete description of the characteristic polynomial---both eigenvalues and their multiplicities---for the action of the operator $\Psi_n$ on each homogeneous component of $H$. The eigenvalues are powers of $n$.
The multiplicities are independent of $n$, and in fact only depend on the dimension sequence of $H$. These results apply in particular to the antipode of $H$,
as the case $n=-1$. We obtain closed forms for the generating function of the sequence
of traces of the Adams operators. In the case of the antipode, the generating function bears a particularly simple relationship to the one for the dimension sequence.
In case $H$ is cofree, we give an alternative description for the
characteristic polynomial and the trace of the antipode in terms of certain palindromic words. 
We discuss parallel results that hold for Hopf monoids in species and for $q$-Hopf algebras.
\end{abstract}

\maketitle

\vspace*{-5pt}
\begin{center}\it Dedicated to the memory of Jean-Louis Loday.\end{center}

\tableofcontents

\section*{Introduction}\label{s:intro}

Let $H$ be a Hopf algebra with antipode $\apode\colon H \to H$. If $H$ is commutative or cocommutative, then it is well-known that $\apode$ is an involution ($\apode^2=\id$). 
In particular, its eigenvalues are $\pm1$. Alternatively, if $H$ is finite-dimensional, then 
$\apode$ has finite even order (and its eigenvalues can be arbitrary even roots of unity). This paper studies  the behavior of the antipode when $H$ is graded connected. 
We prove in Corollary~\ref{c:main} that, in this case, the eigenvalues of $\apode$ are always $\pm1$, even though $\apode$ may have infinite order on any homogeneous component of $H$. 

It is both natural and convenient to consider a more general family of operators on $H$:
the convolution powers of the identity. These are the Adams operators $\Psi_n$,
and the antipode is $\Psi_{-1}$. Our main result, Theorem \ref{t:main}, provides a complete description of the characteristic polynomial for the operator $\Psi_n$ acting on the $m$-th homogeneous component of $H$. For each scalar $n$ and nonnegative integer $m$,
the polynomial is uniquely
determined by the dimension sequence of $H$. The
eigenvalues are powers of $n$.
The multiplicities count samples with replacement by length and weight, with the samples taken from a weighted set that arises, numerically, as the inverse Euler transform of the
dimension sequence of $H$, and algebraically, as a basis of the graded Lie algebra
of primitive elements of a Hopf algebra canonically associated to $H$.

Corollaries~\ref{c:trace-power} and~\ref{c:trace-antipode-gen} provide
information on the trace of the antipode. The former provides a semicombinatorial
description for its values, and the latter the following remarkable expression for
the generating function:
\[
\sum_{m\geq 0} \tr\bigl(\apode\big\vert_{H_m}\bigr)\, t^m = \frac{h(t^2)}{h(t)},
\]
where $h(t)$ is the generating function for the dimension sequence of $H$.
In Corollary~\ref{c:asym} we derive information about the asymptotic behavior of the
sequence of traces.

To put these results into perspective, consider the Hopf algebra of symmetric functions. Calculating the trace of the antipode on suitable linear bases
yields a simple proof of the following interesting identity 
(see Section~\ref{ss:sym}):
\[
c(m) = e(m) - o(m),
\]
where $c(m)$ denotes the number of self-conjugate partitions of $m$, and $e(m)$ and $o(m)$ denote the number of partitions of $m$ with an even number
of even parts, and with an odd number of even parts, respectively.
The generating function for the trace of the antipode in the preceding paragraph
may be seen as an extension of this result to arbitrary graded connected Hopf algebras.

Consider now the Hopf algebra of quasisymmetric functions. 
Another quick calculation reveals that the trace of the antipode on the $m$-th component is a signed count of palindromic compositions of $m$; see Section~\ref{ss:qsym}. 
Corollary~\ref{c:pal} provides a similar result that applies to any graded connected
Hopf algebra that, as a graded coalgebra, is cofree.

Some of the main results admit extensions to certain graded connected $q$-Hopf algebras. Among these, we highlight Corollary~\ref{c:pal-q}, which describes the 
trace of the antipode as a polynomial in $q$,
and Corollary~\ref{c:genfun-trace-q}, which involves
appropriate $q$-generating functions and replaces
the above relationship between the sequence of antipode traces and the dimension sequence.

The theory of Hopf monoids in species often runs in parallel to the theory of
graded connected Hopf algebras. The main results of the paper admit variants
that apply in this context. The sequence of antipode traces is now determined 
from the dimension sequence as follows (Corollary~\ref{c:trace-antipode-gen-sp}): 
the exponential generating function of the former
is the reciprocal of that of the latter.

The paper is organized as follows. In Section \ref{s:prelims}, we discuss the 
necessary preliminaries from Hopf algebra theory.
The proof of Theorem \ref{t:main} is carried out in Section \ref{s:main}. 
Section~\ref{s:trace} focuses on the trace of the Adams operators and the antipode
particularly.
In Section \ref{s:cofree}, we give alternatives to our main results about the antipode that hold in the presence of cofreeness assumptions. Section \ref{s:apps} provides illustrations of the results and some simple calculations. In Appendix~\ref{s:species} we present the results for
 Hopf monoids in species, and in
 Appendix \ref{s:q-Hopf} we treat the case of $q$-Hopf algebras. 

The present paper supersedes and considerably expands on the results of
our extended abstract~\cite{AguLau:2013}.

\section{Preliminaries on Hopf algebras}\label{s:prelims}

Throughout, all vector spaces are over a field $\kk$ of characteristic zero. 

The structure maps of a bialgebra $H$ are denoted by
\begin{gather*}
\mu\colon H\otimes H \to H, \quad  \Delta\colon H \to H\otimes H,\\
\iota\colon \kk \to H, \quad \varepsilon\colon H \to \kk.
\end{gather*}
The \emph{antipode} of a Hopf algebra $H$ is denoted by 
\[
\apode \colon H \to H.
\]

\subsection{Convolution}\label{ss:conv}

Let $H$ be a bialgebra and 
let $\Conv{H}$ denote the space of linear maps $\Map:H\to H$.
The \demph{convolution product} of $P,Q\in\Conv{H}$ is
\[
P\ast Q := \mu\circ(P\otimes Q)\circ \Delta.
\] 
This turns the space $\Conv{H}$ into an associative algebra. The unit element is
\[
\iota\circ\varepsilon.
\]
The bialgebra $H$ is a \demph{Hopf algebra} if and only if the identity of $H$ is convolution-invertible.
In this case, the antipode $\apode$ is the convolution-inverse of the identity map:
\[
\apode\ast\id = \id\ast \apode= \iota\circ\varepsilon.
\] 

Let $H$ be a bialgebra. 
Put $\Delta^{(0)} := \id$, $\Delta^{(1)} := \Delta$, and $\Delta^{(n)} := (\Delta\otimes \id^{\otimes (n-1)})\circ\Delta^{(n-1)}$ for all $n\geq 2$. The superscript is one less than the number of tensor factors in the codomain. Similarly, $\mu^{(n)}$ denotes the map that multiplies $n+1$
elements of $H$, with $\mu^{(0)}:=\id$.  
The convolution powers of any $\Map\in\Conv{H}$ can be written as follows:
\[
	\Map^{\ast0}=\iota\circ\varepsilon \qand \Map^{\ast n} = \mu^{(n-1)} \circ \Map^{\otimes n} \circ \Delta^{(n-1)} \ \ 
\hbox{(for $n\geq1$)}.
\]

\subsection{Adams operators}\label{ss:Adams}

Let $H$ be a Hopf algebra. The convolution powers $\id^{\ast n}$ of the identity of $H$ are defined for any integer $n$. They are called \demph{Adams operators} and are denoted by
\begin{equation}\label{eq:adams}
\Psi_n := \id^{\ast n}:H\to H.
\end{equation}
For $n\geq 1$, we have
\begin{equation}\label{eq:adams2}
\Psi_n = \mu^{(n-1)}\circ\Delta^{(n-1)}.
\end{equation}
Note that $\Psi_0=\iota\circ\varepsilon$ and $\Psi_{-1}=\apode$. Also,
\begin{equation}\label{eq:adams3}
\Psi_{-n} = \apode^{\ast n}
\end{equation}
for all $n$.

This terminology is used in~\cite[\S 3.8]{Car:2007} and~\cite[\S 4.5]{Lod:1992}.
Other common terminology for these operators are \emph{Hopf powers}~\cite{NgSch:2008}, \emph{Sweedler powers}~\cite{KMN:2012,KSZ:2006}, and
\emph{characteristic operations}~\cite{GerSch:1991,Pat:1993}. The paper~\cite[\S 13]{AguMah:2013} studies analogous operators in the context of Hopf monoids in species.

The main goal of this paper is to analyze the characteristic polynomial of these operators
when the Hopf algebra $H$ is graded connected.

\subsection{Coradical filtration and primitive elements}\label{ss:corad}

For more details on the notions reviewed in this section, see~\cite[Ch. 5]{Mon:1993} or~\cite[Ch. 4]{Rad:2012}.

Let $H^{(0)}$ denote the \demph{coradical} of a bialgebra $H$, and let 
\[
H^{(0)} \subseteq H^{(1)} \subseteq H^{(2)}\subseteq \dotsb \subseteq H
\]
denote its coradical filtration. We have
\[
H= \bigcup_{m\geq 0} H^{(m)}.
\]

We say $H$ is \demph{connected} if $H^{(0)}$ is spanned by the unit element $1\in H$.
In this case, $H$ is a Hopf algebra; see Section~\ref{ss:antipode}.
In addition,
$H^{(1)}=H^{(0)} \oplus \mathcal \calP(H)$, where 
\[
\calP(H) := \left\{x\in H \mid \Delta(x) = 1\otimes x + x\otimes 1\right\}
\]
is the space of \demph{primitive elements} of $H$. More generally, setting 
$H_+:=\ker(\epsilon)$ and defining $\Delta_+: H_+\to H_+\otimes H_+$ by
 \[
\Delta_+(x) := \Delta(x)-1\otimes x -x\otimes 1,
\]
we have $H^{(m)}=H^{(0)} \oplus \ker(\Delta_+^{(m)})$, where the iterates of $\Delta_+$ are defined as for $\Delta$.

Let $H$ be a bialgebra.
Let $\gr H$ denote the graded vector space associated to the coradical filtration of $H$:
\begin{gather}
	\gr H :=  H^{(0)} \oplus \left( H^{(1)} \big/ H^{(0)} \right) 
		\oplus \left( H^{(2)} \big/ H^{(1)} \right) \oplus \left( H^{(3)} \big/ H^{(2)} \right) \oplus \dotsb .
\end{gather} 

A filtration-preserving map $f:H\to K$ induces a map
\[
\gr f: \gr H\to \gr K.
\]
The structure maps of $H$ are filtration-preserving (when $H\otimes H$ is endowed with the
tensor product of the coradical filtrations of each factor). The induced maps turn
$\gr H$ into a bialgebra. 
If $H$ is a Hopf algebra, so is $\gr H$.
If $H$ is connected, so is $\gr H$. More importantly:

\begin{lemma}\label{l:foissy}
If the bialgebra $H$ is connected, then the bialgebra $\gr H$ is commutative.
\end{lemma}
This is an immediate consequence of~\cite[Thm. 11.2.5.a]{Swe:1969}.
A direct proof is given in~\cite[Prop.~1.6]{AguSot:2005a}. 
We are thankful to Akira Masuoka, Susan Montgomery and the referee for pointing out the reference to Sweedler's text.

The passage $H\mapsto \gr H$ is functorial with respect to filtration-preserving maps.
It follows that convolution products are preserved:
if $f$ and $g:H\to H$ are linear maps, then
\begin{equation}\label{eq:gr-conv}
\gr(f\ast g) = (\gr f)\ast(\gr g).
\end{equation}

A morphism of bialgebras $f:H\to K$ preserves the coradical filtrations.
The induced map $\gr f:\gr H\to\gr K$ is a morphism of bialgebras.

\subsection{Antipode and {E}ulerian idempotents}\label{ss:antipode}

Any connected bialgebra is a Hopf algebra with antipode 
\begin{equation}\label{eq:s-def}
	\apode = \sum_{k\geq0} \bigl(\iota\circ\varepsilon-\id\bigr)^{\ast k}.
\end{equation}
This basic result can be traced back to Sweedler~\cite[Lem. 9.2.3]{Swe:1969}
and Takeuchi~\cite[Lem. 14]{Tak:1971}; see also~\cite[Lem. 5.2.10]{Mon:1993} and~\cite[Lem. 7.6.2]{Rad:2012}. It follows by expanding
\[
x^{-1} = \frac{1}{1-(1-x)} = \sum_{k\geq 0} (1-x)^k
\] 
in the convolution algebra, with $x=\id$ and $1=\iota\circ\varepsilon$.
 Connectedness guarantees that the sum in \eqref{eq:s-def} is finite when evaluated on any $h\in H$. More precisely, if $h\in H^{(m)}$, then
 $(\id-\iota\circ\varepsilon)^{\ast k}(h) = 0$ for all $k>m$.

Assume for the remainder of the section that $H$ is a connected Hopf algebra.

The binomial theorem yields the following expression for the $n$-th Adams operator,
for all integers $n$:
\begin{equation}\label{eq:adams-pol}
\Psi_n = \sum_{k\geq 0} \binom{n}{k}(\id-\iota\circ\varepsilon)^{\ast k}.
\end{equation}
Moreover, the right-hand side of~\eqref{eq:adams-pol} is well-defined for all scalar values of $n$. In this manner, the Adams operators $\Psi_n$ are defined for all scalars $n$.

Similarly, the following
series expansion defines an element $\log(\id)$ in the convolution algebra:
\begin{align}
\label{eq:log}	\log(\id) & := -\sum_{k\geq 1} \frac1k(\iota\circ\varepsilon-\id)^{\ast k}. 
\end{align}
Additionally, consider the elements $\Eul k$, for $k\geq0$, given by
\begin{gather}\label{eq:eul def}
	\Eul 0 := \unit\circ\counit, \quad \Eul 1 := \log(\id), \quad \Eul k := \frac1{k!}\bigl(\Eul1\bigr)^{\ast k} \ \ (\hbox{for }k>1).
\end{gather}

In case $H$ is commutative or cocommutative,  the  $\Eul k$ form a {complete orthogonal system} of idempotent operators on $H$. 
That is, 
\begin{gather}
	\id = \sum_{k\geq0}\Eul k, \quad
	\Eul k\circ \Eul k = \Eul k, \qand
	\Eul j\circ \Eul k = 0 \hbox{ \ for $j\neq k$.}
\end{gather} 
The $\Eul k$ are the \demph{higher Eulerian idempotents}; $\Eul{1}$ is the
 \demph{first} Eulerian idempotent.  
It follows from~\eqref{eq:adams},~\eqref{eq:eul def}, and the identity $x^{\ast n} = \exp(n\log(x))$ that
\begin{gather}\label{eq:id=eul}
	\Psi_n = \sum_{k\geq0} n^k \Eul k
\end{gather}
for all scalars $n$. In particular,
\[
\apode = \sum_{k\geq0} (-1)^k \Eul k.
\]

If $H$ is cocommutative, $\Eul k$ projects onto the subspace spanned by $k$-fold products of primitive elements of $H$. In particular,
$\Eul1$ projects onto $\calP(H)$. 

For proofs of these results, see \cite[Ch. 4]{Lod:1992}, \cite{Pat:1993} or~\cite[\S{9}]{Sch:1994}. Some instances of these operators in the recent literature include \cite{DPR:1}, \cite{NPT:2013}, and \cite{PatSch:2006}. 
For references to earlier work on Eulerian idempotents, see \cite[\S 14]{AguMah:2013}.

\subsection{Hopf-Lie theory}\label{ss:hopf-lie}

Let $H$ be a bialgebra. The space $\calP(H)$ is a Lie subalgebra of $H$ under the commutator bracket.
Write $\frakg$ for the Lie algebra $\calP(H)$. If $H$ is connected and cocommutative, the Cartier--Milnor--Moore (CMM) theorem yields a canonical isomorphism of Hopf algebras
\[
H\cong \calU(\frakg)
\]
between $H$ and the enveloping algebra of $\frakg$. 

Let $\calS(V)$ denote the symmetric algebra on a space $V$. It carries a unique Hopf algebra
structure for which the elements of $V$ are primitive.
The Poincar\'e--Birkhoff--Witt (PBW) theorem furnishes canonical isomorphisms
\[ 
\calU(\frakg) \cong \calS(\frakg) \qand \gr \calU(\frakg) \cong \calS(\frakg),
\]
where $\frakg$ is an arbitrary Lie algebra.
The former is an isomorphism of coalgebras, the latter of Hopf algebras. 
Here $\calS(\frakg)$ is the symmetric (Hopf) algebra on the vector space underlying $\frakg$.

Proofs of these classical results can be
found in~\cite[\S 3.8]{Car:2007},~\cite[\S~7]{MilMoo:1965} and \cite[App.~B]{Qui:1969}. 
For additional references, see~\cite[Thm. V.2.5]{Kas:1995},~\cite[\S 3.3.4 \& App. A]{Lod:1992}, and~\cite[Thm. 5.6.5]{Mon:1993}.

Lemma~\ref{l:foissy} provides a construction of a commutative connected Hopf algebra $\gr H$ from an arbitrary connected Hopf algebra $H$. This will enable us to employ CMM and PBW in a wider setting than that of
(co)commutative Hopf algebras.

\subsection{Graded bialgebras}\label{ss:graded}

The bialgebra $H$ is \demph{graded} if there is given a vector space decomposition
\[
H=\bigoplus_{m\geq0}H_m
\] 
such that
\begin{gather*}
\mu(H_p\otimes H_q) \subseteq H_{p+q} \text{ for all $p,q\geq 0$}, \quad \Delta(H_m) \subseteq \bigoplus_{p+q=m} H_p \otimes H_q \text{ for all $m\geq 0$,}\\ 
\im(\iota)\subseteq H_{0}, \qand H_m\subseteq \ker(\varepsilon) \text{ for all $m>0$}.
\end{gather*}
(The condition on the unit map $\iota$ simply states that $1\in H_{0}$.)
If $H$ is Hopf and the antipode satisfies
\[
\apode(H_m)\subseteq H_m
\]
for all $m\geq 0$, we say $H$ is a graded Hopf algebra.

If $H$ is an arbitrary bialgebra, then $\gr H$ is a graded bialgebra for which the component of 
degree $n$ is $H^{(n)} \big/ H^{(n-1)}$. 

Let $H$ be a graded bialgebra. The space $\calP(H)$ is then graded with
$\calP(H)_m=\calP(H)\cap H_m$. Moreover, $\calP(H)$ is then a graded Lie algebra.
Similarly, each subspace $H^{(n)}$ is graded with
$(H^{(n)})_m=H^{(n)}\cap H_m$. Hence,
$\gr H$ inherits a second grading for which 
\begin{equation}\label{eq:secgr}
(\gr H)_m := (H^{(0)})_m \oplus\Bigl( (H^{(1)})_m \big/ (H^{(0)})_m \Bigr) \oplus \Bigl((H^{(2)})_m \big/ (H^{(1)})_m\Bigr) \dotsb.
\end{equation}
Moreover,
$H^{(0)}\subseteq H_0$. 

We say that $H$ is \demph{graded connected} if $\dim H_0=1$. The preceding implies that
in this case $H_0=H^{(0)}$, and therefore $H$ is indeed connected. 

Assume that $H$ is a graded connected bialgebra. One may show by induction that $H_m\subseteq H^{(m)}$ for all $m$. It follows that the sum~\eqref{eq:secgr} stops at $(H^{(m)})_m \big/ (H^{(m-1)})_m$.
It also follows, from~\eqref{eq:s-def}, that $H$ is a graded Hopf algebra and
\begin{equation}\label{eq:s-def-graded}
	\apode\big\vert_{H_m} = \sum_{k=0}^m \bigl(\iota\circ\varepsilon-\id\bigr)^{\ast k}\big\vert_{H_m}.
\end{equation}
More generally, it follows from~\eqref{eq:adams-pol} that
\begin{equation}\label{eq:adams-pol-graded}
\Psi_n\big\vert_{H_m} = \sum_{k= 0}^m \binom{n}{k}(\id-\iota\circ\varepsilon)^{\ast k}\big\vert_{H_m}.
\end{equation}

\section{Characteristic polynomials of the Adams operators}\label{s:main}

This section contains the main result (Theorem~\ref{t:main}), which determines
the characteristic polynomials of the Adams operators on a graded connected Hopf algebra 
$H$. These only depend on the dimension sequence of $H$. A number of consequences
about the antipode are also presented. Some preliminaries on enumeration of multisets
are reviewed first.

\subsection{Sampling with replacement and the symmetric algebra}\label{ss:sampling}

Given a sequence $\underline{g}=(g_i)_{i\geq 1}$ and a partition $\lambda = 1^{k_1}2^{k_2}\dotsb r^{k_r}$, put
\begin{equation}\label{eq:binom-g}
  \binom{\underline{g}}{\lambda} := \binom{g_1+k_1-1}{k_1} \dotsb \binom{g_r+k_r-1}{k_r} .
\end{equation}
Let $\abs{\lambda}:=k_1+2k_2+\cdots+rk_r$ denote the size of $\lambda$, and
$\length{\lambda}:=k_1+k_2+\cdots+k_r$  denote the number of parts of $\lambda$. 
Given nonnegative integers $k$ and $m$, set
\begin{gather}\label{eq:eul}
  \mul(k,m) := \sum_{\substack{\abs{\lambda}=m \\ \length\lambda = k}} \binom{\underline{g}}{\lambda}.
\end{gather}
In particular,
\begin{equation}\label{eq:eul0}
\mul(0,m) = \delta(0,m) \qand \mul(k,m)=0 \text{ for all $k>m$.} 
\end{equation}
The numbers $\mul(k,m)$ depend on the given sequence, although this is not reflected in the notation.

If we sample with replacement from a set with $g_i$ elements of \emph{weight} $i$, then 
$\binom{\underline{g}}{\lambda}$ counts the number of samples with weight distribution $\lambda$; in other words, the number of multisets of cardinality
$\length{\lambda}$ and containing exactly $k_i$ elements of weight $i$ for each $i=1,\ldots,r$.
The numbers $\mul(k,m)$ then count the number of multisets of cardinality $k$ and total weight $m$.

Now let $W$ be a positively graded vector space, and let $g_i:=\dim W_i$ for each $i\geq 1$.
Let $W^k$ denote the $k$-th symmetric power of $W$. In other words, $W^k$ is the subspace of the symmetric algebra $\calS(W)$ spanned by $k$-fold products of elements of $W$.
It inherits a grading from $W$, where $(W^k)_m$ is spanned by products $w_1\cdots w_k$ with $\deg(w_1)+\cdots+\deg(w_k)=m$.

Fix a homogeneous basis of $W$, and let it be our sample set. The set of monomials of length $k$ (that is, $k$-fold products of basis elements of $W$) is then a basis of $W^k$.
A multiset of cardinality $k$ and total weight $m$ corresponds to a monomial of length $k$
and degree $m$.
Therefore, 
\[
 \dim (W^k)_m  = \mul(k,m).
\]
The bivariate generating series for $\calS(W)$ by length and degree is then
\begin{equation}\label{eq:sym-gen}
\sum_{k,m\geq 0} \mul(k,m)\, s^k t^m = \prod_{i\geq 1} (1-st^i)^{-g_i}.
\end{equation}
This follows by expanding the right-hand side with the aid of the binomial theorem and employing~\eqref{eq:binom-g} and~\eqref{eq:eul}.

The numbers $\mul(k,m)$ enter in Theorem~\ref{t:main} below.

\subsection{Characteristic polynomial}\label{ss:char}

We state some standard results from linear algebra.

\begin{lemma}\label{l:eigenvalues}
Let $V$ be a finite-dimensional vector space and $\Map\in\End(V)$ a linear transformation.	
\begin{enumerate}
\item[(i)] Let $U$ be a $\Map$-invariant subspace of $V$.
If $\bMap\in\End(V/U)$ denotes the linear transformation induced by $\Map$ on the quotient, and $\uMap\in\End(U)$ denotes the restriction of $\Map$ to $U$, then the characteristic polynomials of these three maps satisfy
\[
	\chi_\Map(x) = \chi_{\uMap}(x) \, \chi_{\bMap}(x).
\]
\item[(ii)] The characteristic polynomials of $\Map$ and of the dual map $\dMap\in\End(V^*)$ are equal. \qed
\end{enumerate}
\end{lemma}

We are now ready for our main result. 
Let $H$ be a graded connected Hopf algebra.
We assume from this point onwards that the homogeneous components $H_m$ of $H$ are finite-dimensional. 
We let $\tilH$ denote the \demph{graded dual} of $\gr H$ with respect to the grading~\eqref{eq:secgr}. It is a graded bialgebra with homogeneous components
\[
\tilH_m := \bigl((\gr H)_m\bigr)^*.
\]
If $H$ is graded connected, $\tilH$ is graded connected and cocommutative, by Lemma~\ref{l:foissy}. We let $\frakg:=\calP(\tilH)$ denote the graded Lie
algebra of primitive elements of $\tilH$. 
For each $i\geq 1$, let 
\[
g_i := \dim \frakg_i
\]
denote the dimension of the homogeneous component of $\frakg$ of degree $i$.
Consider the corresponding numbers $\mul(k,m)$, as in~\eqref{eq:eul}.

\begin{theorem}\label{t:main}
Let $H$ be as above.
For every scalar $n$ and nonnegative integer $m$, the characteristic polynomial of 
the restriction $\Psi_n\big\vert_{H_m}$ of the $n$-th Adams operator is
\begin{equation}\label{eq:char poly}
	\chi\bigl(\Psi_n\big\vert_{H_m}\bigr)(x) = \prod_{k=0}^m(x-n^k)^{\mul(k,m)}.
\end{equation} 
\end{theorem}

Before the proof, a few remarks are in order. First, note that the factor indexed by $k=0$ is nontrivial only when $m=0$, according to~\eqref{eq:eul0}. Second, note that the exponents
$\mul(k,m)$ do not depend on $n$.
Additional information on the exponents $\mul(k,m)$ is provided in Proposition~\ref{p:main}. 

\begin{proof}
First of all, it suffices to establish~\eqref{eq:char poly} when $n$ is a nonnegative integer.
Indeed, both sides depend polynomially on $n$; the left-hand side in view of~\eqref{eq:adams-pol-graded}.

We argue that we may replace $H$ with $\gr H$. Indeed, since $\Psi_n$ preserves both the grading and the coradical filtration of $H$, 
it preserves the filtration
\[
 (H^{(0)})_m\subseteq (H^{(1)})_m \subseteq \cdots\subseteq (H^{(m)})_m = H_m
\]
 for each $m$. 
By repeated application of Lemma \ref{l:eigenvalues}(i) we deduce that
  \[
  \chi\bigl(\Psi_n\big\vert_{H_m}\bigr) 
  = \chi\bigl(\gr(\Psi_n)\big\vert_{(\gr H)_m}\bigr).
  \]
  In addition, by~\eqref{eq:gr-conv}, 
  \[
  \gr(\Psi_n) = \gr(\id^{\ast n})=(\gr \id)^{\ast n} = \id^{\ast n} = \Psi_n.
  \] 
  Therefore,
  \[
   \chi\bigl(\Psi_n\big\vert_{H_m}\bigr) 
  = \chi\bigl(\Psi_n\big\vert_{(\gr H)_m}\bigr).
  \]

Next, we may replace $\gr H$ with $\tilH$. Indeed, the map $\Map\mapsto\Map^*$ is an isomorphism of convolution algebras $\End(H)\cong\End(H^*)$ (where duals and endomorphisms are in the graded sense). Together with Lemma \ref{l:eigenvalues}(ii) this implies that
 \[
 \chi\bigl(\Psi_n\big\vert_{(\gr H)_m}\bigr) 
  = \chi\bigl(\Psi_n\big\vert_{\tilH_m}\bigr).
\]

Now let $\frakg = \calP(\tilH)$. By CMM, $\tilH \cong \calU(\frakg)$, and by PBW,
$\gr \calU(\frakg) \cong \calS(\frakg)$ as Hopf algebras.  The same argument as above shows that we may replace $\tilH$ with $\calS(\frakg)$. 

As $\calS(\frakg)$ is cocommutative, the Eulerian idempotents are available. 
From~\eqref{eq:id=eul} we have that
\[
\chi\bigl(\Psi_n\big\vert_{\calS(\frakg)_m}\bigr) = \prod_{k\geq0}  \chi\bigl(n^k\Eul k\big\vert_{\calS(\frakg)_m}\bigr).
\] 
It thus suffices to calculate the characteristic polynomial of the $\Eul k$ on $\calS(\frakg)$. 

Finally, the action of $\Eul k$ on $\calS(\frakg)$ is simply projection onto $\frakg^k$, the subspace spanned by $k$-fold products of elements of $\frakg$. 
It follows that
\[
{\chi\bigl(n^k \Eul k\big\vert_{\calS(\frakg)_m}\bigr)}(x) = (x-n^k)^{\mul(k,m)},
\]
where 
\[
\mul(k,m) = \dim\, (\frakg^k)_m.
\]
This completes the proof.
\end{proof}

\begin{remark}
One may easily see that the Adams operators act on $\calS(W)$ as follows
\[
	\Psi_n(w_1\dotsb w_k) = n^k w_1\dotsb w_k,
\]
where $w_i\in W$, $i=1,\ldots,k$.  The proof of Theorem~\ref{t:main} can then be
completed without explicit mention of the Eulerian idempotents. 

On the other hand, assume that $H$ is a graded connected Hopf algebra
that is either commutative or cocommutative.
The expression~\eqref{eq:id=eul} for the Adams operators in terms of the
Eulerian idempotents shows that the former are simultaneously diagonalizable.
The forthcoming thesis of Amy Pang~\cite{Pan:2014} contains a discussion
of a common eigenbasis for the Adams operators on such $H$.
\end{remark}

The exponents $\mul(k,m)$ are determined by the dimension sequence of $\frakg$, through \eqref{eq:binom-g} and~\eqref{eq:eul}. In turn, this sequence is related to the dimension
sequence of $H$ by
\begin{equation}\label{eq:alt eul}
	1+\sum_{m\geq1} h_m t^m = \prod_{i\geq1}\bigl(1-t^i\bigr)^{-g_i},
\end{equation}
where $h_m:=\dim H_m$. Indeed, the right-hand side is the generating series for 
$\calS(\frakg)$, and we have from the proof of Theorem~\ref{t:main}
that $H_m\cong \tilH_m \cong \calS(\frakg)_m$ as vector spaces.
It follows that the sequences $(g_i)$ and $(h_m)$ determine each other.
In particular:

\begin{proposition}\label{p:main}
The exponents $\mul(k,m)$ 
are determined by the dimension sequence of $H$. 
\end{proposition}

\begin{remark}
Equation~\eqref{eq:alt eul} says that the sequence $(h_m)_{m\geq 1}$ is the \demph{Euler transform} of $(-g_i)_{i\geq 1}$, in the sense of~\cite[p. 20]{SloPlo:1995}. 
The nonnegativity of the sequence $(g_i)$ restricts the class of sequences $(h_m)$ that may be realized as dimension sequences of graded connected Hopf algebras.
\end{remark}

\subsection{Eigenvalues of the antipode}\label{ss:eigen-antipode}

Let $H$ be a graded connected Hopf algebra and $\apode$ its antipode.

Applying Theorem~\ref{t:main} in the case $n=-1$ yields information about
the antipode, since $\apode=\Psi_{-1}$. We obtain
\begin{equation}\label{eq:char-apode}
	\chi\bigl(\apode\big\vert_{H_m}\bigr)(x) = (x-1)^{\emul(m)} (x+1)^{\omul(m)},
\end{equation} 
where 
\[
\emul(m) := \sum_{\text{$k$ even}} \mul(k,m)
\qand
\omul(m) := \sum_{\text{$k$ odd}} \mul(k,m).
\]
In particular:

\begin{corollary}\label{c:main}
The eigenvalues of the antipode are $\pm1$.
\end{corollary}

\begin{remark}
Corollary \ref{c:main} fails for general Hopf algebras. Let $\omega$ be a primitive cube root of unity and consider Taft's Hopf algebra $T_3(\omega)$~\cite{Taf:1971}, with generators $\{g, x\}$, and relations $\{g^3=1$, $x^3=0$, $gx=\omega\,xg\}$.
The coproduct and antipode are determined by $\Delta(g)=g\otimes g$, $\apode(g)=g^{-1}$,
$\Delta(x)=1\otimes x+x\otimes g$, and $\apode(x)=-xg^{-1}$. Here $x^2+\omega\, x^2 g$
is an eigenvector of $\apode$ with eigenvalue $\omega$.
\end{remark}

From Corollary~\ref{c:main} we deduce:

\begin{corollary}\label{c:diag}
The antipode is diagonalizable if and only if it is an involution.
\end{corollary}

\begin{remark}
The antipode of a graded connected Hopf algebra need not be an involution 
(or equivalently, diagonalizable).
For example, consider the Malvenuto-Reutenauer Hopf algebra of permutations (Example \ref{eg:MR}). Its antipode is of infinite order already on the homogeneous component of degree $3$; see Remark 5.6 in \cite{AguSot:2005}. 
\end{remark}

\subsection{Composition powers of the antipode}\label{ss:comp-pow}

Consider $\apode^2$, the composition of the antipode $\apode$ with itself.
Corollary~\ref{c:main} implies that $1$ is the only eigenvalue of $\apode^2$, so $\apode^2-\id$ is
nilpotent on each homogeneous component.
We have the following more precise result.
\begin{proposition}\label{p:even} 
The map $(\apode^2-\id)\vert_{H_m}$ is
nilpotent of order at most $m$.
\end{proposition} 
\begin{proof}
Since $\gr(H)$ is commutative, $\apode^2 =\id$ on $\gr(H)$. Hence
$(\apode^2 - \id)(H^{(m)}) \subseteq H^{(m-1)}$ for all $m\geq 1$.
On $H^{(1)}=\kk\oplus \calP(H)$, we have $\apode=\pm \id$, and then $\apode^2-\id = 0$.
By induction, $(\apode^2 - \id)^m (H^{(m)}) = 0$ for all $m\geq 1$.
The statement follows by recalling that $H_m \subseteq H^{(m)}$.
\end{proof}
 
\begin{remark} 
Let $d_m$ be the order of nilpotency of $(\apode^2-\id)\vert_{H_m}$, so that $1\leq d_m\leq m$
by Proposition~\ref{p:even}.
The lower bound $d_m=1$ is attained by any involutory Hopf algebra, for all $m$.
Computations suggest that the Hopf algebra of \emph{signed permutations}~\cite{BonHoh:2006} attains the upper bound $d_m=m$ for all $m\geq 1$.
\end{remark}

\begin{example}\label{eg:MR}
Consider the Malvenuto--Reutenauer Hopf algebra $H=\ssym$~\cite{MalReu:1995}.
We claim that $d_m\leq m-1$ for $m\geq 2$ (and $d_1=1$).
Let $\id_m=123\ldots m$ be the identity permutation, in one-line notation, and let $\omega_m$ be the longest permutation of $m$ elements, $\omega_m=m\ldots 321$. 
Let $\mathcal{F}$ and $\mathcal{M}$ denote the \emph{fundamental} and
\emph{monomial} bases of $H$, in the notation of~\cite{AguSot:2005}. It follows from~\cite[Cor. 6.3]{AguSot:2005} that 
\[
H_m = \bigl(H_m\cap H^{(m-1)}\bigr) \oplus K_m,
\]
where $K_m$ is the one-dimensional space spanned by $\mathcal{F}_{\omega_m}=\mathcal{M}_{\omega_m}$.
The proof of Proposition~\ref{p:even} shows that on $H_m\cap H^{(m-1)}$ the order of nilpotency is at most $m-1$. On the other hand, it is easy to see that
\[
\apode(\mathcal{F}_{\omega_m}) = (-1)^m \mathcal{F}_{\id_m} \qand
\apode(\mathcal{F}_{\id_m}) = (-1)^m \mathcal{F}_{\omega_m}.
\]
Therefore, $\apode^2$ is the identity on $K_m$. The claim follows. 
\end{example}

We turn to higher composition powers of the antipode. Since $H^{\mathrm{op}}$ is another graded connected Hopf algebra, it possesses an antipode, and this is the inverse of $\apode$ by \cite[Lem. 1.5.11]{Mon:1993} or \cite[Prop. 1.23]{AguMah:2013}.
In particular, $\apode$ is invertible, and we may consider powers $\apode^n$ for any integer $n$.

\begin{proposition}\label{p:comp-pow}
For any integer $n$,
\begin{equation}\label{eq:comp-pow}
	\chi\bigl(\apode^n\big\vert_{H_m}\bigr)(x) = 
	\begin{cases}
\chi\bigl(\id\big\vert_{H_m}\bigr)(x) & \text{ if $n$ is even,} \\
\chi\bigl(\apode\big\vert_{H_m}\bigr)(x)         & \text{ if $n$ is odd.}
\end{cases}
\end{equation} 
\end{proposition}
\begin{proof}
As in the proof of Theorem~\ref{t:main}, we may assume that $H$ is commutative.
In this case, $\apode^2=\id$ and hence $\apode^n=\id$ for even $n$ and $\apode^n=\apode$ for odd $n$.
\end{proof}


\section{The trace of the Adams operators}\label{s:trace}

We study the trace of the Adams operators on $H$. The generating functions for their
sequences of traces admit closed expressions in terms of the inverse Euler transform of the dimension sequence of $H$. 
The generating function for the trace of the antipode is particularly remarkable.

\subsection{Generating functions for the trace}\label{ss:gen-trace}

We return to the situation of Theorem~\ref{t:main}. Thus, $H$ is a
graded connected Hopf algebra, and the integers $\mul(k,m)$ are
determined by its dimension sequence through \eqref{eq:binom-g}, \eqref{eq:eul},
and \eqref{eq:alt eul}.

As an immediate consequence of this theorem, we have:

\begin{corollary}\label{c:trace-power}
For all scalars $n$ and nonnegative integers $m$,
\begin{equation}\label{eq:trace-power}
	\tr\bigl(\Psi_n\big\vert_{H_m}\bigr) = \sum_{k=0}^m n^k \mul(k,m)
\end{equation}
In particular,
\begin{equation}\label{eq:trace-antipode}
	\tr\bigl(\apode\big\vert_{H_m}\bigr) =  \sum_{k=0}^m (-1)^k\mul(k,m)  =  \emul(m) - \omul(m). 
\end{equation}
\end{corollary}

We turn to generating functions for the trace. First we consider the
$n$-th Adams operator. Recall that the sequence $(g_i)$ is determined by the
dimension sequence of $H$ through~\eqref{eq:alt eul}.

\begin{corollary}\label{c:trace-power-gen}
For all scalars $n$,
\begin{equation}\label{eq:trace-power-gen}
\sum_{m\geq 0} \tr\bigl(\Psi_n\big\vert_{H_m}\bigr)\, t^m = 
\prod_{i\geq 1} (1-nt^i)^{-g_i}.
\end{equation}
\end{corollary}
\begin{proof}
This follows at once from~\eqref{eq:sym-gen} and~\eqref{eq:trace-power}.
\end{proof}

For each scalar $n$, let
\[
h_n(t) := \sum_{m\geq 0} \tr\bigl(\Psi_n\big\vert_{H_m}\bigr)\, t^m 
\]
denote the generating function for the sequence of traces of $\Psi_n$.
As a consequence of Corollary~\ref{c:trace-power-gen},
these functions satisfy certain interesting relations. In order to state them, let
\[
\mu_k := \{\omega\in\bC \mid \omega^k=1\}
\]
denote the group of complex $k$-th roots of unity. 

\begin{proposition}\label{p:trace-power-rel}
For each scalar $n$ and positive integer $k$,
\begin{equation}\label{eq:trace-power-rel}
h_{n^k}(t^k) = \prod_{\omega\in\mu_k} h_{\omega n}(t).
\end{equation}
In particular,
\begin{equation}\label{eq:trace-power-rel2}
h_{n^2}(t^2) = h_n(t)\, h_{-n}(t).
\end{equation}
\end{proposition}
\begin{proof}
We employ the factorization
\[
1-t^k = \prod_{\omega\in\mu_k}(1-\omega t).
\]
Then, from~\eqref{eq:trace-power-gen}, we have
\[
h_{n^k}(t^k) = \prod_{i\geq 1} (1-n^kt^{ki})^{-g_i} = 
\prod_{i\geq 1}\prod_{\omega\in\mu_k} (1-\omega n t^i)^{-g_i} = 
\prod_{\omega\in\mu_k}
h_{\omega n}(t).\qedhere
\]
\end{proof}

The generating function for the trace of the antipode takes a special form.
Let
\[
h(t) := 1+ \sum_{m\geq 1} h_m\, t^m
\]
denote the generating function for the dimension sequence of $H$. 

\begin{corollary}\label{c:trace-antipode-gen}
\begin{equation}\label{eq:trace-antipode-gen}
\sum_{m\geq 0} \tr\bigl(\apode\big\vert_{H_m}\bigr)\, t^m = \frac{h(t^2)}{h(t)}.
\end{equation}
\end{corollary}
\begin{proof}
Since $\Psi_1=\id$ and $\Psi_{-1}=\apode$, we have
$h(t)=h_1(t)$
and $\sum_{m\geq 0} \tr\bigl(\apode\big\vert_{H_m}\bigr)\, t^m = h_{-1}(t)$. Thus, 
\eqref{eq:trace-antipode-gen} is the case $n=1$ of~\eqref{eq:trace-power-rel2}.
\end{proof}

\begin{remark}
Corollary~\ref{c:trace-antipode-gen} shows that the sequence of antipode traces is determined by the dimension sequence in a simple manner. The result also shows that, conversely, the
dimension sequence is determined by the sequence of antipode traces,
since the relation~\eqref{eq:trace-antipode-gen} can be solved for $h(t)$. 
\end{remark}

\begin{example}\label{eg:geometric}
Suppose the dimension sequence of $H$ is given by $h_m:=r^m$, where $r$ is a fixed nonnegative integer. Then $h(t)=1/(1-rt)$ and
\[
\frac{h(t^2)}{h(t)} = \frac{1-rt}{1-rt^2} = \sum_{n\geq 0} r^n\,t^{2n} - \sum_{n\geq 0} r^{n+1}\,t^{2n+1}.
\]
It follows from~\eqref{eq:trace-antipode-gen} that
\[
\tr\bigl(\apode\big\vert_{H_m}\bigr)= 
\begin{cases}
r^{m/2} & \text{ if $m$ is even,} \\
-r^{(m+1)/2}         & \text{ if $m$ is odd.}
\end{cases}
\]

We have computed the trace of the antipode without knowing anything other than the dimension sequence of $H$. A Hopf algebra with the given dimension sequence is the free algebra on $r$ primitive generators of degree $1$; a direct computation of the trace 
can then be carried out. We do this for the dual Hopf algebra in Example~\ref{eg:geometric2}.
\end{example}

We record the generating function for the trace of the composition powers of the antipode.

\begin{corollary}\label{c:comp-pow}
For any integer $n$,
\begin{equation}\label{eq:comp-pow-gen}
	\sum_{m\geq 0} \tr\bigl(\apode^n\big\vert_{H_m}\bigr)\, t^m =  
	\begin{cases}
h(t) & \text{ if $n$ is even,} \\
h(t^2)/h(t)      & \text{ if $n$ is odd.}
\end{cases}
\end{equation} 
\end{corollary}
\begin{proof}
This follows from Proposition~\ref{p:comp-pow} and Corollary~\ref{c:trace-antipode-gen}.
\end{proof}

\subsection{The trace of the antipode versus the dimension sequence}\label{ss:fun-m}

In this section, we write
\[
a_m := \tr\bigl(\apode\big\vert_{H_m}\bigr)
\]
for each nonnegative integer $m$. Let $a(t)$ be its generating function. 

We analyze the behavior of the sequence $(a_m)$ in relation
to the dimension sequence $(h_m)$.

\begin{corollary}\label{c:fun-m}
If the sequence $(h_m)$ satisfies a linear recursion with constant coefficients, then so does 
$(a_m)$.
\end{corollary}
\begin{proof}
We employ~\cite[Thm. 4.1.1]{Sta:2012}.
In this situation, the series $h(t)$ is rational, and hence so is $a(t)$, by~\eqref{eq:trace-antipode-gen}.
\end{proof}

We turn to asymptotics. We assume there exists a meromorphic function $h(z)$ of a complex variable $z$,
holomorphic on a neighborhood of $0$, and such that its Taylor expansion is the generating function $h(t)$. 
It then follows from Pringsheim's Theorem~\cite[Thm. IV.6]{FlaSed:2009} that a dominant singularity occurs at a positive real number $R$. Moreover, $R\leq 1$. Indeed, if $R>1$, the coefficients $h_m$ would approach $0$, by the Exponential Growth Formula~\cite[Thm. IV.7]{FlaSed:2009}, and this would force the integers $h_m$ to be $0$ from a point on. 
(As we remark at the end of Section~\ref{ss:schur},
this can only happen if $H$ is the one-dimensional Hopf algebra, a triviality which we exclude from consideration.)

\begin{corollary}\label{c:asym}
Suppose that $R$ is the unique singularity of $h(z)$
in the disk $\abs{z}\leq R^{1/4}$. Let $\gamma$ be the order of this singularity.
Suppose further that $h(z)$ is nonzero in the disk $\abs{z}\leq R^{1/2}$, and
\[
h(-R^{1/2}) \neq \pm h(R^{1/2}).
\]
Then
\begin{equation}\label{eq:asym}
\frac{a_m}{h_m} \sim \frac{R^{m/2}}{2^{\gamma}}\,
\Bigl(\frac{1}{h(R^{1/2})} + (-1)^m\,\frac{1}{h(-R^{1/2})}\Bigr).
\end{equation}
\end{corollary}
\begin{proof}
The hypotheses guarantee that $R$ is the unique dominant singularity for $h(z)$,
and also that $\pm R^{1/2}$ are the unique dominant singularities for $a(z)=h(z^2)/h(z)$.
We then have the standard approximations~\cite[\S B.IV]{FlaSed:2009}, \cite[\S 5.2]{Wil:2006}
\[
h_m \sim \frac{1}{\Gamma(\gamma)}\, m^{\gamma-1} R^{-m} h^*(R)
\]
and
\[
a_m \sim \frac{2^\gamma}{\Gamma(\gamma)}\, m^{\gamma-1} R^{-m/2} h^*(R) \Bigl(\frac{1}{h(R^{1/2})} + (-1)^m\,\frac{1}{h(-R^{1/2})}\Bigr),
\]
where $\Gamma$ is the gamma function and ${\displaystyle h^*(R)= \lim_{z\to R} \Bigl(1-\frac{z}{R}\Bigr)^{\gamma}\, h(z)}$.
The result follows. 
\end{proof}

\begin{example}\label{eg:fib}
Suppose the dimension sequence is given by
\[
h_0=h_1=h_2:=1 \qand h_{m} := h_{m-1}+h_{m-2} \text{ for all $m\geq 3$.}
\]
Thus, for $m\geq 1$, $h_m$ are the Fibonacci numbers.
In this case,
\[
h(z) = \frac{z^2-1}{z^2+z-1} = (z^2-1)(z+\phi)^{-1}(z-1/\phi)^{-1},
\]
where $\phi=(1+\sqrt{5})/2$ is the golden ratio. The hypotheses of Corollary~\ref{c:asym}
are satisfied with $R=1/\phi$ and $\gamma=1$. We obtain from~\eqref{eq:asym}
the approximation
\[
\frac{a_m}{h_m} \sim 
\begin{cases}
\phi^{-m/2} & \text{ if $m$ is even,}\\
-\phi^{(-m+3)/2} & \text{ if $m$ is odd.}
 \end{cases}
\]

A Hopf algebra with this dimension sequence is discussed in Section~\ref{ss:peakqsym},
and the sequence $a_m$ is computed explicitly; see~\eqref{eq:apode-peak}. The above may then be seen to
follow from the well-known approximation $h_m\sim \phi^m/\sqrt{5}$
for the Fibonacci numbers.
\end{example}

\subsection{Schur indicators}\label{ss:schur}

A theme occurring in the recent Hopf algebra literature involves a
generalization of the \demph{Frobenius--Schur indicator function} of a finite group. 
If $\rho\colon G \to \mathrm{End}(V)$ is a complex representation of $G$, then the (second) indicator is
\[
	\nu_2(G,\rho) := \frac1{|G|}\sum_{g\in G} \tr\rho(g^2).
\]
The only values this invariant can take on irreducible representations are $0,1,-1$, and this occurs precisely when $V$ is a complex, real, or quaternionic representation, respectively~\cite[Prop. 39]{Ser:1977}. In \cite{LinMon:2000}, a reformulation of the definition was given in terms of convolution powers of the \emph{integral}%
\footnote{A construct present for finite-dimensional Hopf algebras that is unavailable for general graded connected Hopf algebras.}
 in $\bC G$. This extended the notion of (higher) Schur-indicators to all finite-dimensional Hopf algebras, and has since become a valuable tool for the study of these algebras \cite{KMM:2002, NgSch:2008, SagVeg:2012, Shi:2012}. In case $\rho$ is the regular representation (and $H$ is semisimple), it is shown in \cite{KSZ:2006} that the higher Schur indicators can be reformulated further, removing all mention of the integral: for all nonnegative integers $n$,
\[
	\nu_{n+1}(H) = \tr(\apode\circ \Psi_n).
\]
See also \cite{KMN:2012}. 
Our results lead to the
following formula for these invariants in case $H$ is graded connected (instead of finite-dimensional). 

\begin{corollary}\label{c:schur}
Let $H$ be a graded connected Hopf algebra. Then, for all scalars $n$ and nonnegative integers $m$,
\[
	 \tr(\apode\circ \Psi_n) = \sum_{k=0}^m (-n)^k \mul(k,m),
\]
where $\mul(k,m)$ is as in Theorem~\ref{t:main}. 
\end{corollary}
\begin{proof}
As in the proof of Theorem \ref{t:main}, we may assume that $H$ is commutative. Then $\apode$ is a morphism of algebras, and we have
\[
\apode\circ \Psi_n = \apode\circ \mu^{(n-1)}\circ\Delta^{(n-1)} = \mu^{(n-1)}\circ \apode^{\otimes n} \circ \Delta^{(n-1)} = \apode^{\ast n} = \Psi_{-n}.
\] 
We used~\eqref{eq:adams2} and~\eqref{eq:adams3}.
The result follows from~\eqref{eq:trace-power}.
\end{proof}

\begin{remark}
We mention in passing that the only Hopf algebra $H$ that is at the same time connected
and finite-dimensional is the (unique) one-dimensional Hopf algebra. Indeed, 
combining Lemma~\ref{l:foissy} with CMM and PBW (as in the proof of Theorem~\ref{t:main}) we have that $H\cong\calS(V)$ as vector spaces, for some space $V$. But the only space $V$ for which the symmetric algebra is finite-dimensional is $V=0$. Hence $H\cong\kk$.
The situation is of course different over fields of positive characteristic or for $(-1)$-Hopf algebras.
\end{remark}

\section{The case of cofree graded connected Hopf algebras}\label{s:cofree}

We study the characteristic polynomial and the trace of the antipode of a
graded connected Hopf algebra that, as a graded coalgebra, is cofree. Since the
former are invariant under duality, the results apply as well to graded connected Hopf algebras that are free as algebras. (We make no further mention of this point as we proceed.)

\subsection{Cofreeness}\label{ss:cofree}

A graded connected Hopf algebra $H$ is \demph{cofree} if,
 as a graded coalgebra, it is isomorphic to a \demph{deconcatenation}
coalgebra $\cT(V)$ on a graded vector space $V$. 
The underlying space of the latter is 
\[
\cT(V) := \bigoplus_{k\geq 0} V^{\otimes k}
\]
(the same as that of the tensor algebra $\calT(V)$). The coproduct on a $k$-fold tensor is
\[
\Delta(x_1\cdots x_k) = 1\otimes (x_1\cdots x_k)+\sum_{i=1}^{k-1} (x_1\cdots x_i) \otimes (x_{i+1}\cdots x_k) + (x_1\cdots x_k)\otimes 1.
\]
For more details, see~\cite[\S 2.6]{AguMah:2010} or~\cite[\S 4.5]{Rad:2012}.

In this situation, we have
\[
H^{(m)} \cong \bigoplus_{k=0}^m V^{\otimes k} \qand \calP(H)\cong V
\]
as graded vector spaces. 

The \demph{shuffle product} of two tensors $x_1\ldots x_i$ and $x_{i+1}\ldots x_k$
is the following element of $\cT(V)$:
\[
\sum_{\sigma} x_{\sigma^{-1}(1)}\ldots x_{\sigma^{-1}(k)},
\]
where the sum is over all permutations $\sigma\in S_k$ such that
\[
\sigma(1)<\cdots<\sigma(i) \qand \sigma(i+1)<\cdots<\sigma(k).
\]
The shuffle product turns the deconcatenation coalgebra $\cT(V)$ 
into a commutative graded connected Hopf algebra. The antipode acts on a tensor 
by reversing the components:
\begin{equation}\label{eq:ant-shuffle}
\apode(x_1x_2\dotsb x_k) = (-1)^k x_k\dotsb x_2x_1.
\end{equation}

Under the assumption of cofreeness, the following stronger version
of Lemma~\ref{l:foissy} is available.

\begin{lemma}\label{l:as05}
Let $H$  be a graded connected Hopf algebra that is cofree as a graded coalgebra.
Then 
\[
\gr H \cong \cT(V)
\]
as graded Hopf algebras, where $V=\calP(H)$.
\end{lemma}
This result appears in \cite[Prop. 1.5]{AguSot:2005}.

\subsection{Palindromes and Lyndon words}\label{ss:pal-lyn}

A (weighted) \demph{alphabet} is a set that is graded by the positive integers and whose  homogeneous components are finite. The elements of degree $n$
 are called \demph{letters of weight} $n$. A \demph{word} is a sequence of letters in the alphabet and a \demph{palindrome} is a word 
that coincides with its reversal.
The \demph{length} of a word is the number of letters, and the \demph{weight} of a word is the sum 
of the weights of its letters.
 
Given an alphabet, let 
\[
\pal(m) \qand \nopal(m)
\] 
denote the number of palindromes of weight $m$, and the number of non-palindromic words
of weight $m$, respectively. Also, let
\[
\epal(m) \qand \opal(m)
\] 
denote the number of palindromes of weight $m$ of even and odd length, respectively,
and let 
\[
\pal(k,m)
\] 
denote the number of palindromes of length $k$ and weight $m$.

\smallskip

Assume now that a cofree graded connected Hopf algebra $H$ is given.
Thus, $H\cong\cT(V)$ as graded coalgebras, with $V=\calP(H)$.
Let
\[
v_n := \dim V_n
\]
denote the dimension of the space of homogeneous primitive elements of degree $n$.
Let $h_n$ and $g_n$ be as in Proposition~\ref{p:main}, so $(h_n)$ is the dimension
sequence of $H$ and $(g_n)$ is the dimension sequence of $\calP(\tilH)$.
Since $H\cong \cT(V)$ as graded vector spaces, we have
\begin{equation}\label{eq:geometric}
1+\sum_{m\geq 1} h_m t^m = \frac{1}{1-\sum_{n\geq 1} v_n t^n}.
\end{equation}
Together with~\eqref{eq:alt eul}, this yields
\begin{equation}\label{eq:Witt}
1- \sum_{n\geq 1} v_n t^n =  \prod_{i\geq1} \bigl(1-t^i\bigr)^{g_i}.
\end{equation}
In particular, the sequences $(h_n)$, $(g_n)$, and $(v_n)$ determine each other.

Fix a homogeneous basis of $V$, and let it be our alphabet. Thus,
there are $v_n$ letters of weight $n$. Equation~\eqref{eq:geometric} then says that
 $h_m$ is the total number of words of weight $m$. In particular,
\[
\nopal(m) = h_m - \epal(m) - \opal(m).
\] 
Equation~\eqref{eq:Witt} says that $g_i$ is the number of \emph{Lyndon words} of weight $i$ in the given alphabet. 
Then equation~\eqref{eq:eul} says that $\mul(k,m)$ counts the number of multisets of Lyndon words of cardinality $k$ and total weight $m$.
\emph{Witt's formula}~\cite[Thm. 2.2]{KanKim:1996} provides
an explicit formula for $(g_n)$ in terms of $(v_n)$:
\begin{equation}\label{eq:Witt2}
g_n = \sum_{d|n} \frac{ \mu(d)}{d} \sum_{\lambda\,\vdash\, n/d} \frac{\bigl(\ell(\lambda)-1\bigr)!}{\lambda!} v^\lambda,
\end{equation}
where $\mu(d)$ is the classical M\"obius function, the inner sum is over all partitions $\lambda=1^{k_1}2^{k_2}\cdots r^{k_r}$ of $n/d$, $\ell(\lambda)=k_1+k_2+\cdots+k_r$, $\lambda!=k_1! k_2!\cdots k_r!$, and 
\[
v^\lambda := v_1^{k_1} v_2^{k_2}\cdots v_r^{k_r}.
\]
(The special case of~\eqref{eq:Witt2} in which all letters are of weight $1$ appears in
 \cite[Cor. 5.3.5]{Lot:1997} and  \cite[Cor. 4.14]{Reu:1993}.) 

\subsection{Characteristic polynomial and trace of the antipode in the cofree case}\label{ss:char-cofree}

As an alternative to~\eqref{eq:char poly} and~\eqref{eq:trace-antipode},
we have the following expressions for the characteristic polynomial and trace of the antipode of a cofree graded connected Hopf algebra $H$. 

\begin{theorem}\label{t:pal}
For $H$ as above and any nonnegative integer $m$,
\begin{equation}\label{eq:S pal}
  \chi\bigl(\apode\big\vert_{H_m}\bigr)(x) = (x+1)^{\opal(m)}(x-1)^{\epal(m)}(x^2-1)^{\nopal(m)/2}.
\end{equation}
\end{theorem}

\begin{proof}
By Lemma~\ref{l:as05}, $\gr H$ is isomorphic to the Hopf algebra $\cT(V)$, where
the latter is equipped with the shuffle product and the deconcatenation coproduct.
As in Theorem~\ref{t:main}, it suffices to analyze the antipode $\apode$ of this Hopf algebra.

Now, it follows from~\eqref{eq:ant-shuffle} that each palindrome yields an eigenvector of $\apode$.
The eigenvalue is $\pm 1$ according to the parity of the length. This explains the first two factors in~\eqref{eq:S pal}. The non-palindromic words pair up with their reversals and organize in $2\times 2$ blocks of the
form
\[
\pm\begin{pmatrix}
 0 &  1\\  1 & 0
\end{pmatrix}
\]
where the sign again depends on the parity of the length. This accounts for the remaining factor. 
\end{proof}

As an immediate consequence, we have:

\begin{corollary}\label{c:pal}
For $H$ as above and any nonnegative integer $m$,
\begin{equation}\label{eq:S trace}
  \tr\bigl(\apode\big\vert_{H_m}\bigr) = \epal(m) - \opal(m) = \sum_{k=0}^m (-1)^k\pal(k,m).
\end{equation}
\end{corollary}

Since there are no palindromes of even length and odd weight, \eqref{eq:S pal} and \eqref{eq:S trace} imply
\[
 \chi\bigl(\apode\big\vert_{H_m}\bigr)(x) = (x+1)^{\pal(m)} (x^2-1)^{\nopal(m)/2}
\]
and
\[
\tr\bigl(\apode\big\vert_{H_m}\bigr) =  - \pal(m)
\]
for all odd $m$.

\begin{example}\label{eg:geometric2}
Let $V$ be an $r$-dimensional vector space. We view it as a graded vector space
concentrated in degree $1$ and consider the Hopf algebra $\cT(V)$. Our alphabet
consists of $r$ letters of weight $1$, and the palindrome distribution is
\[
\pal(k,m) = \begin{cases}
r^k & \text{ if $m=2k$,} \\
r^{k+1}         & \text{ if $m=2k+1$,}\\
0 & \text{ otherwise.}
\end{cases}
\]
It follows from~\eqref{eq:S trace} that
\[
\tr\bigl(\apode\big\vert_{H_m}\bigr)= 
\begin{cases}
r^{m/2} & \text{ if $m$ is even,} \\
-r^{(m+1)/2}         & \text{ if $m$ is odd.}
\end{cases}
\]
We arrived at the same conclusion by different means in Example~\ref{eg:geometric}.
\end{example}

We return to the general discussion. {}From~\eqref{eq:trace-antipode} and~\eqref{eq:S trace} we deduce
\begin{equation}\label{eq:pal-eul}
  \epal(m) - \opal(m)  =  \emul(m) - \omul(m),
\end{equation}
or equivalently,
\begin{equation}\label{eq:pal-eul2}
  \sum_{k=0}^m (-1)^k\pal(k,m)  =  \sum_{k=0}^m (-1)^k\mul(k,m).
\end{equation}
In general, the pairs $(\epal,\opal)$ and $(\emul,\omul)$,
as well as the triangular arrays $\pal$ and $\mul$, are different.

\begin{example}\label{eg:MR2}
Consider again the Malvenuto--Reutenauer Hopf algebra $\ssym$.  
We compare the integers $\mul(k,m)$ and $\pal(k,m)$ for low values of $k$ and $m$. 

$\ssym$ is cofree and the relevant alphabet is the set of permutations with no global descents; see \cite[Cor. 6.3]{AguSot:2005}. 
On the component of degree $m=3$, we have the following distribution of palindromes.
\begin{center}
\begin{tabular}{l@{\quad}|@{\quad}c@{\qquad}c@{\qquad}c}
length ($k$) & 1  &  2  &  3  \\
\hline
permutations & $123, 132, 213$  &  $231,312$  &  $321$ \\
words on alphabet  & $123, 132, 213$  &  $12\cmrg1, 1\cmrg12$  &  $1\cmrg1\cmrg1$ \\
\hline
$\pal(k,3)$ & 3 & 0 & 1
\end{tabular}
\end{center}
Beneath each permutation, we recorded its expression as words in the alphabet. 
Counting those words that are palindromic we obtained the integers $\pal(k,3)$.

The integer $v_m$ is the number of permutations of $m$ elements with no global descents. 
The integers $(g_m)$ are calculated from either \eqref{eq:alt eul}
or~\eqref{eq:Witt}. The first few values are as follows.
\begin{center}
\begin{tabular}{c@{\quad}|@{\quad}c@{\quad}c@{\quad}c@{\quad}c@{\quad}c@{\quad}c}
$m$ & 1  &  2  &  3  & 4 & 5 & 6\\
\hline
$v_m$ & $1$  &  $1$  &  $3$ & $13$ &  $71$ & $461$\\
$g_m$  & $1$  &  $1$  &  $4$  & $17$ & $92$ & $572$ \\
\end{tabular}
\end{center}
The sequences $(v_m)$ and $(g_m)$ are  A003319 and A112354 in \cite{oeis}, respectively.
Finally, the integers $\mul(k,m)$ are computed from \eqref{eq:eul}. 
For $m=3$ we find the following.
\begin{center}
\begin{tabular}{c@{\quad}|@{\quad}c@{\quad}c@{\quad}c}
$k$ & 1  &  2  &  3  \\
\hline
$\mul(k,3)$ & $4$  &  $1$  &  $1$
\end{tabular}
\end{center}

Beyond $m=3$, the integers $\mul(k,m)$ and $\pal(k,m)$ differ more drastically; see Figure \ref{f:triangles}. However, the alternating sum of the entries in each column is the same for both arrays,
as predicted by~\eqref{eq:pal-eul2}.
\begin{figure}[htb]
\centering
\subfigure[The array $\mul(k,m)$.]{
\begin{tabular}[c]{@{\,}c|c@{\ \ \ }c@{\ \ \ }c@{\ \ \ }c@{\ \ \ }c@{\ \ \ }c@{\ \ \ }c@{\,}}
	$k$ $\backslash$ $m$ & 1 & 2 & 3 & 4 & 5 & 6 \\
	\hline
	1 & 1 & 1 & 4 & 17 & 92 & 572 \\
	2 & & 1 & 1 & 5 & 21 & 119 \\
	3 &&& 1 & 1 & 5 & 22 \\
	4 &&&& 1 & 1  & 5 \\
	5 &&&&& 1 & 1 \\
	6 &&&&&& 1 \\
	\end{tabular}
}
\quad\quad
\subfigure[The array $\pal(k,m)$.]{
\begin{tabular}[c]{@{\,}c|c@{\ \ \ }c@{\ \ \ }c@{\ \ \ }c@{\ \ \ }c@{\ \ \ }c@{\ \ \ }c@{\,}}
	$k$ $\backslash$ $m$ & 1 & 2 & 3 & 4 & 5 & 6 \\
	\hline 
	1 & 1 & 1 & 3 & 13 & 71  & 461 \\
	2 &  & 1 & 0 & 1 & 0 & 3 \\
	3 &  &  & 1 & 1 & 4 & 14  \\
	4 &  &  &  & 1 & 0 & 2   \\
	5 &  &  &  &  & 1 & 1  \\
	6 &  &  &  &  &  & 1 \\
	\end{tabular}
}
\caption{The arrays $\mul(k,m)$ and $\pal(k,m)$ for $\ssym$.}
\label{f:triangles}
\end{figure} 
\end{example}

\subsection{Generating functions}\label{ss:genfun}

We continue to employ the notation of Section~\ref{ss:pal-lyn}. Let
\[
v(t) := \sum_{n\geq 1} v_n t^n
\]
be the generating function for the dimension sequence of $V$.

We have the following generating functions for even and odd palindromes.
The functions are bivariate to account for length and weight.
 
\begin{proposition}\label{p:genfun-pal}
\begin{align}
\label{eq:genfun-epal}
\sum_{k,m\geq 0} \pal(2k,m)\,s^k t^m & = \frac{1}{1-s v(t^2)},\\
\label{eq:genfun-opal}
\sum_{k,m\geq 0} \pal(2k+1,m)\,s^k t^m & = \frac{v(t)}{1-s v(t^2)}.
\end{align}
\end{proposition}
\begin{proof}
Consider a palindrome of even length $2k$. Removing the first and last letters (which are equal), yields a palindrome of length $2k-2$. The weights of the two palindromes differ by twice the weight of this letter. Therefore,
\[
\pal(2k,m) = v_1\pal(2k-2,m-2) + v_2\pal(2k-2,m-4)+\cdots.
\]
This recursion leads at once to~\eqref{eq:genfun-epal}. A similar argument establishes~\eqref{eq:genfun-opal}.
\end{proof}

\begin{corollary}\label{c:genfun-trace}
\begin{equation}\label{eq:genfun-trace}
\sum_{m\geq 0} \tr\bigl(\apode\big\vert_{H_m}\bigr)\, t^m = \frac{1-v(t)}{1-v(t^2)}.
\end{equation}
\end{corollary}
\begin{proof}
This follows by subtracting~\eqref{eq:genfun-opal} from~\eqref{eq:genfun-epal}, letting $s=1$,
and employing~\eqref{eq:S trace}.
\end{proof}

\begin{remark}
Corollary~\ref{c:genfun-trace} is a special case of Corollary~\ref{c:trace-antipode-gen},
in view of~\eqref{eq:geometric}. The above may be regarded as a \emph{semicombinatorial} proof of this result, which is possible under the cofreeness assumption.
\end{remark}

\section{Examples}\label{s:apps}

We carry out explicit calculations for the Hopf algebra of symmetric functions and a few related Hopf algebras, focusing on the trace of the antipode. They offer no difficulty,
as explicit formulas for the antipodes of these Hopf algebras are known. Our purpose here is simply to illustrate some of the results from the preceding sections.

The paper \cite{ABS:2006} contains a concise description of each of the Hopf algebras
discussed in this section. Other references are given as we proceed.

\subsection{Symmetric functions}\label{ss:sym}

Consider the Hopf algebra of symmetric functions $H = \sym$. See \cite[Ch.I]{Mac:1995} for the results used below. 
On the basis of Schur functions, the antipode acts by  
$\apode(s_{\lambda}) = (-1)^{\abs{\lambda}}s_{\lambda'}$, where $\lambda'$ is the partition conjugate to $\lambda$.
Therefore,
\begin{equation}\label{eq:apode-sym}
\tr\bigl(\apode\big\vert_{H_m}\bigr) = (-1)^m c(m),
\end{equation}
where $c(m)$ is the number of self-conjugate partitions of $m$. 

We turn to Corollary~\ref{c:trace-power}.
For this Hopf algebra, $g_i=1$ for all $i\geq1$. 
Hence $ \binom{\underline{g}}{\lambda} = 1$ for all $\lambda$, and $\mul(k,m) = p_k(m)$, the number of partitions of $m$ into $k$ parts. {}From~\eqref{eq:trace-antipode} we deduce
\begin{equation}\label{eq:self-conj}
(-1)^m c(m)   = \sum_{k=0}^m (-1)^k p_k(m).
\end{equation}
(Note that $p_0(m)=0$ for $m>0$.) 
The number of odd parts in a partition of $m$ has the same parity as $m$. Hence, the previous identity is equivalent to
\begin{equation}\label{eq:self-conj2}
c(m) = e(m) - o(m),
\end{equation}
where $e(m)$ and $o(m)$ denote the number of partitions of $m$ with an even number
of even parts, and with an odd number of even parts, respectively.
This identity appears in~\cite[Ex. 1.60]{Aig:2007} and \cite[Ch. 1, Ex. 22.b]{Sta:2012}.

It is possible to obtain this result more directly as follows.
Consider the power sum basis of $\sym$. Since $\apode(p_\lambda) = (-1)^{\ell(\lambda)}p_\lambda$, we have  
\begin{equation}\label{eq:apode-sym2}
\tr(\left.\apode\right\vert_{H_m}) = p_e(m) - p_o(m),
\end{equation}
where $p_e(m)$ and $p_o(m)$ are the number of partitions of $m$ of even length
and of odd length, respectively.
Equating~\eqref{eq:apode-sym} and~\eqref{eq:apode-sym2} gives~\eqref{eq:self-conj} again.

We further illustrate Corollary~\ref{c:trace-power} by deriving certain
identities involving the Littlewood--Richardson coefficients $c_{\mu,\nu}^{\,\lambda}$.
Recall that the latter are the structure constants for both the product and coproduct 
on the Schur basis of $\sym$:
\[
	 s_\mu \cdot s_\nu = \sum_\lambda c_{\mu,\nu}^{\,\lambda} \, s_\lambda 
	 \quad\qand\quad
	 \Delta(s_\lambda)  = \sum_{\mu,\nu} c_{\mu,\nu}^{\,\lambda} \, s_\mu\otimes s_\nu.
\] 
Formula~\eqref{eq:trace-power} (with $n=\pm 2$) yields the following identities,
for all $m\geq 1$:
\begin{gather}
\sum_{\lambda,\mu,\nu \vdash m} (c_{\mu,\nu}^{\,\lambda})^2 \ = \
\sum_{k=1}^m \,2^k\, p_k(m)
\end{gather}
and
\begin{gather}
 \sum_{\lambda,\mu,\nu \vdash m} c_{\mu,\nu}^{\,\lambda}c_{\mu',\nu'}^{\,\lambda} \ = \ 
 \sum_{k=1}^m \,(-1)^{m-k}\, 2^k\, p_k(m).
\end{gather}
 Incidentally, the fact that the antipode preserves (co)products says that $c_{\mu,\nu}^{\,\lambda}=c_{\mu',\nu'}^{\,\lambda'}$.

\subsection{Schur $P$-functions}\label{ss:oddsym}

A partition of an integer is \demph{strict} if its parts are all distinct.
It is \demph{odd} if each of its parts is odd.

Let $\lambda$ be a strict  partition and $P_\lambda\in\sym$ 
the corresponding \emph{Schur $P$-function}, as in \cite[III.8]{Mac:1995}.
Let $H$ be the subspace of $\sym$ spanned by the $P_\lambda$, as $\lambda$
runs over all strict partitions. Then $H$ is a Hopf subalgebra of $\sym$. We have
\begin{equation}\label{eq:apode-strict}
	\apode(P_\lambda) = (-1)^{\abs{\lambda}} P_\lambda.
\end{equation}
Therefore, 
\[
	\tr(\left.\apode\right\vert_{H_m}) = (-1)^m p_d(m),
\]
where $p_d(m)$ is the number of strict partitions of $m$. 

It is known that $H$ is the subalgebra of $\sym$ 
generated by the odd power sums $p_{2i+1}$, $i\geq 0$. Therefore,
\begin{equation}\label{eq:apode-odd}
\tr(\left.\apode\right\vert_{H_m}) = (-1)^m p_o(m),
\end{equation}
where $p_o(m)$ is the number of odd partitions of $m$. Equating~\eqref{eq:apode-strict}
and~\eqref{eq:apode-odd} recovers the classical fact that odd and strict partitions are equinumerous~\cite[Prop. 1.8.5]{Sta:2012}.

Regarding the quantities in Proposition~\ref{p:main}, we have that $g_i=1$ if $i$ is odd
and $0$ otherwise.
It follows that $\binom{\underline{g}}{\lambda} = 1$ when $\lambda$ is 
odd and $\binom{\underline{g}}{\lambda} = 0$ otherwise. 
Therefore, $\mul(k,m)$ is the number of 
odd partitions of $m$ of length $k$. In an odd partition, the parities of $m$ and $k$
are the same. Thus, identity~\eqref{eq:trace-antipode} simply counts
odd partitions according to their length.

\subsection{Quasisymmetric functions}\label{ss:qsym}

Let us turn to the Hopf algebra $H=\qsym$ of quasisymmetric functions.
Consider the \emph{fundamental} and \emph{monomial}
quasisymmetric functions, denoted by $F_\alpha$ and 
 $M_\alpha$, respectively. As $\alpha$ runs over the compositions of $m$,
 both $\{F_\alpha\}$ and $\{M_\alpha\}$ constitute bases of $H_m$. 
 
 The antipode has the following descriptions:
\[
	\apode(F_\alpha) = (-1)^m F_{\rev{\alpha}'} 
	\qand 
	\apode(M_\alpha) = (-1)^{\length{\alpha}} \sum_{\beta \leq \alpha} M_{\rev{\beta}},
\]
where $\rev \gamma$ is the reversal of $\gamma$, $\gamma'$ is its conjugate 
(obtained by reflecting the ribbon diagram of $\gamma$ across the main diagonal), and $\leq$ is the refinement partial order on compositions. 
Note that $\alpha = \rev{\alpha}'$ if and only if $\alpha$ is symmetric with respect to reflection across the anti-diagonal. There are precisely $2^{(m-1)/2}$ of these when $m$ is odd, and zero when $m$ is even. Calculating the trace on the fundamental basis 
we thus obtain
\begin{gather}\label{eq:qsym}
\tr(\left.\apode\right\vert_{H_m}) = \begin{cases}
-2^{(m-1)/2} & \text{ if $m$ is odd,} \\
  0       & \text{ otherwise.}
\end{cases}
\end{gather}
The compositions $\alpha$ that contribute to the trace on the monomial basis
satisfy $\rev{\alpha}\leq\alpha$. Since reversal is an order-preserving involution,
this happens if and only if $\rev{\alpha}=\alpha$, that is, 
when $\alpha$ is palindromic. Let $\pal(m)$
 denote the number of palindromic compositions of $m$.
 If $m$ is even, exactly half of the palindromic compositions of $m$ have odd length; 
 if $m$ is odd, all of them do. We conclude that 
\begin{gather}\label{eq:qsym2}
\tr(\left.\apode\right\vert_{H_m}) = \begin{cases}
-\pal(m) & \text{ if $m$ is odd,} \\
  0       & \text{ otherwise.}
\end{cases}
\end{gather}
One may arrive at the same identity from~\eqref{eq:S trace}.

Equating~\eqref{eq:qsym} and~\eqref{eq:qsym2} we deduce that, for all odd $m$,
\[
 \pal(m) = 2^{(m-1)/2}.
\]
It is easy to give a direct proof of this fact (and of $\pal(m) = 2^{\lfloor m/2\rfloor}$ for all $m\geq0$).

$\qsym$ is cofree, so Theorem~\ref{t:pal} applies. The space of
primitive elements is spanned by the monomials $M_{(n)}$, for $n\geq 1$.
One finds that 
\[
	\pal(k,m) = \begin{cases}
		\displaystyle\binom{\lceil m/2 \rceil-1}{\lceil k/2\rceil-1} &\hbox{if $m$ is even, or if $m$ is odd and $k$ is odd}, \\[1.5ex]
		\hskip2.5em 0 &\hbox{if $m$ is odd and $k$ is even}.
		\end{cases}
\]
Formula~\eqref{eq:S trace} boils down in this case to the basic identities $2^h=\sum_{j=0}^h \binom{h}{j}$ and $0^h = \sum_{j=0}^h (-1)^j\binom{h}{j}$. 

Regarding the quantities in Proposition~\ref{p:main}, we have by~\eqref{eq:Witt} that $g_i$ is the
number of Lyndon words of weight $i$ in an alphabet with one letter of weight $n$ for each $n\geq 1$. This number is given by $g_1=1$ and
\[
g_i = \frac{1}{i} \sum_{d|i} \mu(d)\,2^{i/d}
\]
for $i\geq 2$~\cite[Prop. 2.3]{KanKim:1996}.

\subsection{Peak quasisymmetric functions}\label{ss:peakqsym}

Let $H$ be the Hopf algebra of \emph{peak} quasisymmetric functions \cite{Ste:1997}. It is 
a Hopf subalgebra of $\qsym$, with a basis $\theta_\alpha$ indexed by compositions $\alpha$ into odd parts (\demph{odd compositions}). The number of odd compositions of $m$ is the Fibonacci number $f_m$ (with $f_0=f_{1}=f_2=1$).

A formula for the antipode of $H$ is given in \cite{BHvW:2003}: 
\[
	\apode(\theta_\alpha) = (-1)^{m}\theta_{\rev \alpha}. 
\]
It follows that 
\begin{equation}\label{eq:apode-peak}
	\tr(\left.\apode\right\vert_{H_m}) = \begin{cases} 
		f_{m/2}, & \hbox{if $m$ is even,} \\ 
		-f_{\lceil m/2 \rceil +1},& \hbox{if $m$ is odd,}
		\end{cases}
\end{equation}
(as palindromic odd compositions of $m$ arise from odd compositions of $m/2$). 
One may arrive at the same identity from~\eqref{eq:S trace}.

As for $\qsym$, $H$ is cofree. There is one primitive element of degree $n$ for
each odd $n$. One finds that 
\[
	\pal(k,m) = \begin{cases}
		\displaystyle \hskip.25em \binom{(m+k)/4-1}{(m-k)/4} & \hbox{if $m$ is even and $4 \mid (m-k)$,} \\[2ex]
		\displaystyle \binom{\lfloor(m+k-1)/4\rfloor}{\lfloor(m-k+1)/4\rfloor} & \hbox{if $m$ and $k$ are odd,} \\[2ex] 
		\hskip3.5em 0 & \hbox{otherwise.}
		\end{cases}
\]
Information about these numbers can be found in \cite[A046854 and A168561]{oeis}.

Formula~\eqref{eq:S trace} yields the following basic identities:
\[
f_h = \sum_{j=0}^{\lfloor h/2\rfloor} \binom{h-j-1}{j}  \ \hbox{for $h\geq 0$} \ \ \qand \ \ 
f_h = \sum_{j=0}^{h-2} \binom{\lfloor\frac{h+j}{2} \rfloor-1}{\lfloor\frac{h-j}{2}\rfloor-1} \ \hbox{for $h\geq 2$.}
\]

\appendix
\section{Hopf monoids in species}\label{s:species}

The results from the earlier sections admit variants for Hopf monoids in species.
We list the main ones in this section, along with indications for the proofs,
which are similar to the ones for graded connected Hopf algebras.

This section assumes familiarity with the notion of Hopf monoid in species,
as developed in~\cite{AguMah:2010,AguMah:2013}
and with the notation employed there. For the most part, the latter (shorter)
reference suffices. The antipode is discussed in~\cite[\S 5]{AguMah:2013}
and the Adams operators (convolution powers of the identity)
appear in~\cite[\S 14.4]{AguMah:2013}.

All Hopf monoids $\tH$ are assumed to be connected and finite-dimensional.
That is, $\tH[\emptyset]$ is one-dimensional, and for each finite set $I$, the vector space $\tH[I]$ is finite-dimensional.

\subsection{Characteristic polynomial}\label{ss:char-sp}

The starting point is the following result, whose proof is similar to that of~\cite[Prop.~1.6]{AguSot:2005a}.

\begin{lemma}\label{l:foissy-sp}
The Hopf monoid $\gr\tH$ (associated to the coradical filtration of $\tH$) 
is commutative.
\end{lemma}

Let $h_m:=\dim\tH[m]$, and let
\[
\egfh(t) := 1+ \sum_{m\geq 1} h_m\, \frac{t^m}{m!}
\]
denote the exponential generating function for the dimension sequence of $\tH$.

\begin{theorem}\label{t:main-sp}
For every scalar $n$ and finite set $I$, the characteristic polynomial of 
the restriction $\Psi_n\big\vert_{\tH[I]}$ of the $n$-th Adams operator is of the form
\begin{equation}\label{eq:char poly-sp}
	\chi\bigl(\Psi_n\big\vert_{\tH[I]}\bigr)(x) = \prod_{k=0}^m(x-n^k)^{\expmul(k,m)},
\end{equation} 
where $m=\abs{I}$. The nonnegative integers $\expmul(k,m)$ are independent of $n$, and are determined
by the dimension sequence of $\tH$, as follows:
\begin{equation}\label{eq:sym-gen-sp}
\sum_{k,m\geq 0} \expmul(k,m)\, s^k \frac{t^m}{m!} = \egfh(t)^s.
\end{equation}
\end{theorem}
\begin{proof}
 As in the proof of Theorem~\ref{t:main}, a combination of Lemma~\ref{l:foissy-sp} with the
PBW and CMM theorems (for Hopf monoids in species~\cite[\S 15]{AguMah:2013})
shows that we can assume that
\[
\tH = \calS(\tP)
\]
for a certain positive species $\tP$. Here, $\calS(\tP)$ is 
the free commutative monoid on $\tP$ with its canonical Hopf monoid structure~\cite[\S 7]{AguMah:2013}. 

Let $\egfp(t)$ be the exponential generating function for the dimension sequence of $\tP$.
Since $\calS(\tP)= \wE\circ\tP$,
where $\wE$ is the exponential species, we have
\[
\egfh(t) = \exp\bigl(\egfp(t)\bigr) \qand \egfh(t)^s = \exp\bigl(s\egfp(t)\bigr).
\]

On the other hand, a direct calculation of the Adams operators on
$\calS(\tP)$ shows that the characteristic polynomial is as in~\eqref{eq:char poly-sp}
and that the integers $\expmul(k,m)$ are determined by
\[
\sum_{k,m\geq 0} \expmul(k,m)\, s^k \frac{t^m}{m!} = \exp\bigl(s\egfp(t)\bigr).
\]
The result follows.
\end{proof}

\subsection{Trace of the antipode}\label{ss:trace-sp}
Let
\[
\egfa(t) := 1+ \sum_{m\geq 1} \tr\bigl(\apode\big\vert_{\tH[m]}\bigr)\, \frac{t^m}{m!}
\]
denote the exponential generating function for the trace of the antipode of a Hopf monoid $\tH$. This is none other than the reciprocal of the exponential generating function for the dimension sequence.

\begin{corollary}\label{c:trace-antipode-gen-sp}
\begin{equation}\label{eq:trace-antipode-gen-sp}
\egfa(t) = \frac{1}{\egfh(t)}.
\end{equation}
\end{corollary}
\begin{proof}
This follows from Theorem~\ref{t:main-sp}, taking $n=-1$ in~\eqref{eq:char poly-sp}
and $s=-1$ in~\eqref{eq:sym-gen-sp}.
\end{proof}

\begin{example}\label{eg:sigma}
Let $\tH=\tSig$ be the Hopf monoid of set compositions~\cite[\S 11.1]{AguMah:2013}.
We have
\[
\egfh(t) = \frac{1}{2-\exp(t)}.
\]
Therefore, $\egfa(t) = 2-\exp(t) = 1-\sum_{m\geq 1}\frac{t^m}{m!}$, and we obtain
\begin{equation}\label{eq:sigma}
\tr\bigl(\apode\big\vert_{\tH[m]}\bigr) = -1
\end{equation}
for all $m\geq 1$. This result can also be obtained by a direct calculation, starting from
either of the expressions for the antipode of $\tSig$ given in~\cite[Prop. 59 or Thm. 60]{AguMah:2013}.
\end{example}

For an extension of this result, assume that there is a positive species $\tP$ such that 
\[
\tH \cong \wL\circ\tP
\] 
as species, where $\wL$ is the species of linear orders. Not every Hopf monoid $\tH$ is of this form, but this is the case if $\tH$
is free or cofree~\cite[\S 6]{AguMah:2013}. In this situation, we have
\begin{equation}\label{eq:S trace-sp}
\tr\bigl(\apode\big\vert_{\tH[m]}\bigr) = - \dim \tP[m].
\end{equation}
Note that if $\tP=\wE$, then $\tH\cong\tSig$, and~\eqref{eq:S trace-sp} recovers~\eqref{eq:sigma}.

\begin{example}\label{eg:pi}
Let $\tH=\tPi$ be the Hopf monoid of set partitions~\cite[\S 9.3]{AguMah:2013}.
We have
\[
\egfh(t) = \exp\bigl(\exp(t)-1\bigr).
\]
Therefore,  
\[
\egfa(t) = \exp\bigl(1-\exp(t)\bigr) = 1-t+\frac{t^3}{3!}+\frac{t^4}{4!}
-2 \frac{t^5}{5!}-9 \frac{t^6}{6!}-9 \frac{t^7}{7!}+50 \frac{t^8}{8!}+ 267 \frac{t^9}{9!}+\cdots.
\]
It follows that 
\begin{equation}\label{eq:pi}
\tr\bigl(\apode\big\vert_{\tH[m]}\bigr) = \Pi_e(m) - \Pi_o(m),
\end{equation}
where $\Pi_e(m)$  and $\Pi_o(m)$ denote the number of set partitions of $[m]$ into an even and an odd number of blocks, respectively.
This result can also be obtained by a direct calculation, starting from
either of the expressions for the antipode of $\tPi$ given in~\cite[Thm. 33 or Prop. 35]{AguMah:2013}.
\end{example}

\begin{remark}
The Hopf monoid $\tPi$ is in many ways parallel to the Hopf algebra $\sym$ of symmetric functions
(Section~\ref{ss:sym}). 
A consequence of the preceding calculation, however, is that $\tPi$ does not admit a linear
basis that behaves under the antipode in the same manner as the Schur basis of $\sym$. More precisely, there is no basis 
$\{s_{\pi} \mid \pi\vdash [m]\}$ of the space $\tPi[m]$ with the property that
\[
\apode(s_{\pi}) = (-1)^m\, s_{\pi'},
\]
for some map $\pi\to\pi'$ on the set of set partitions of $[m]$. Indeed, if this were the case,
the sequence of antipode traces would alternate in sign.
\end{remark}

For an extension of the calculation in Example~\ref{eg:pi}, let $\tH$ be a Hopf monoid and $\tP$ a positive species such that
\[
\tH \cong \wE\circ\tP
\] 
as species. (The first part of the proof of Theorem~\ref{t:main-sp} shows that every connected Hopf monoid is of this form.) In this situation, an $\tH$-structure on a finite set $I$ is an \emph{assembly} of $\tP$-structures, in the sense of~\cite[\S 1.4]{BerLabLer:1998}. Conversely, one may regard a $\tP$-structure as a \emph{connected} $\tH$-structure. We then have
\begin{equation}\label{eq:trace-antipode-sp}
	\tr\bigl(\apode\big\vert_{\tH[m]}\bigr) = h_e(m) - h_o(m),
\end{equation}
where $h_e(m)$ and $h_o(m)$ denote the number of $\tH$-structures with an even and odd number of connected components, respectively. 
If $\tP=\wE$, then \eqref{eq:trace-antipode-sp} recovers \eqref{eq:pi}.

The combination of~\eqref{eq:trace-antipode-gen-sp} and~\eqref{eq:trace-antipode-sp} provide a semicombinatorial description for the reciprocal of any power series arising as the exponential generating function for the dimension sequence of a connected Hopf monoid in species.

\section{$q$-{H}opf algebras}\label{s:q-Hopf}

Fix a scalar $q$. A $q$-Hopf algebra is a Hopf monoid in the lax braided monoidal
category of graded vector spaces, with lax braiding $V\otimes W \to W\otimes V$
given by
\[
x\otimes y\mapsto q^{mn}\, y\otimes x,
\]
where $x\in V$ and $y\in W$ are homogeneous elements of degrees $m$ and $n$.
If $q=1$, a $q$-Hopf algebra is just a graded Hopf algebra as in Section~\ref{ss:graded}.
For information on $q$-Hopf algebras, see~\cite[\S 2.3]{AguMah:2010}.
 
In this section we discuss extensions of some of the main results
from earlier sections to the context of connected $q$-Hopf algebras. 

\subsection{Cofreeness for connected $q$-Hopf algebras}\label{ss:cofree-q}

Let $V$ be a graded vector space. When $x\in V_n$, we write $\abs{x}=n$.
The deconcatenation coalgebra on $V$ (Section~\ref{ss:cofree}),
endowed with the \demph{$q$-shuffle product}, is a connected $q$-Hopf algebra. The $q$-shuffle product of two homogeneous tensors $x_1\ldots x_i$ and $x_{i+1}\ldots x_k$
is the following element of $\cT(V)$:
\[
\sum_{\sigma} q^{\inv_x(\sigma)}\,x_{\sigma^{-1}(1)}\ldots x_{\sigma^{-1}(k)},
\]
where the sum is over all permutations $\sigma\in S_k$ such that
\[
\sigma(1)<\cdots<\sigma(i) \qand \sigma(i+1)<\cdots<\sigma(k),
\]
and 
\[
\inv_x(\sigma) := \sum_{\substack{a<b\\ \sigma(a)>\sigma(b)}} \abs{x_a}\abs{x_b}.
\]
We denote the resulting $q$-Hopf algebra by $\cT_q(V)$.
The antipode is given by
\begin{equation}\label{eq:ant-shuffle-q}
\apode(x_1x_2\dotsb x_k) = (-1)^k q^{\inv_x(k)}\,x_k\dotsb x_2x_1,
\end{equation}
where
\[
\inv_x(k) := \sum_{a<b} \abs{x_a}\abs{x_b}.
\]

The following is an extension of Lemma~\ref{l:as05}. The proof of the latter result given in \cite[Props. 1.4 \& 1.5]{AguSot:2005} yields its extension as well.

\begin{lemma}\label{l:as05-q}
Let $H$ be a connected $q$-Hopf algebra that is cofree as a graded coalgebra. 
Then 
\[
\gr H \cong \cT_q(V)
\] 
as $q$-Hopf algebras, where $V=\calP(H)$.
\end{lemma}

\subsection{Characteristic polynomial and trace of the antipode}\label{ss:char-cofree-q}

Let $H$ be a connected $q$-Hopf algebra that is cofree as a graded coalgebra. 
We lay the groundwork for a description of the characteristic polynomial of such a Hopf algebra.
Let a weighted alphabet be given,
with $v_n$ letters of weight $n$, as in Section~\ref{ss:pal-lyn}. 
The \demph{multiweight} of a word is the sequence
of letter weights. If the word has weight $m$, its multiweight is a composition of 
$m$, and we write $\alpha\vDash m$.  

Given a composition $\alpha$, let
\[
\pal(\alpha) \qand \nopal(\alpha)
\]
denote the number of palindromes and nonpalindromes, respectively, of multiweight $\alpha$. If
$\alpha=(a_1,\ldots,a_k)$, then
\[
\pal(\alpha) = \begin{cases}
\prod_{i=1}^{\lceil k/2 \rceil} v_{a_i} & \text{ if $\alpha=\rev\alpha$,} \\
0         & \text{ otherwise, }
\end{cases}
\qqand
\nopal(\alpha) = \Bigl(\prod_{i=1}^{k} v_{a_i}\Bigr)  - \pal(\alpha).
\]
(Recall that $\rev\alpha=(a_k,\ldots,a_1)$ denotes the reversal of $\alpha$.)
Let $\ell(\alpha)$ denote the length $k$ of $\alpha$, and let
\[
\inv(\alpha) := \sum_{1\leq i<j\leq k} a_i a_j.
\]

Let $H$ be as above. Fix a homogeneous basis of $V=\calP(H)$, 
and take it as our alphabet. Thus, $v_n = \dim V_n$.

\begin{theorem}\label{t:pal-q} 
For each nonnegative integer $m$, the characteristic polynomial of the antipode is
\begin{equation}\label{eq:S pal-q}
 \chi\left(\apode\big\vert_{H_m}\right)(x) = 
  	\prod_{\alpha\,\vDash\, m}\left(x-(-1)^{\ell(\alpha)}q^{\inv(\alpha)}\right)^{\pal(\alpha)} \cdot 
	 \left(x^2-q^{2\,\inv(\alpha)}\right)^{\nopal(\alpha)/2}.
\end{equation}
\end{theorem}
\begin{proof}
Lemma~\ref{l:as05-q} allows us to assume that $H=\cT_q(V)$. The result follows
from~\eqref{eq:ant-shuffle-q}, as in the proof of Theorem \ref{t:pal}. 
\end{proof}

In particular, the eigenvalues of the antipode of such a $q$-Hopf algebra
are positive or negative powers of $q$. We record the resulting expression for the trace.

\begin{corollary}\label{c:pal-q}
\begin{equation}\label{eq:S trace-q}
 \tr\left(\apode\big\vert_{H_m}\right) = 
  	\sum_{\alpha\,\vDash\, m} (-1)^{\ell(\alpha)} \pal(\alpha)\,q^{\inv(\alpha)}.
\end{equation}
\end{corollary}

We recover~\eqref{eq:S pal} and~\eqref{eq:S trace} as the case $q=1$ of~\eqref{eq:S pal-q} and~\eqref{eq:S trace-q}. 

\begin{example}\label{eg:geometric-q}
Let $V$ be an $r$-dimensional vector space. View it as a graded vector space
concentrated in degree $1$ and consider the $q$-Hopf algebra $\cT_q(V)$.
Then
\[
\pal(\alpha) = \begin{cases}
r^{\lceil m/2 \rceil} & \text{ if $\alpha=(1^m)$,} \\
0 & \text{ otherwise.}
\end{cases}
\]
Therefore,
\[
\tr\bigl(\apode\big\vert_{H_m}\bigr)= 
(-1)^m r^{\lceil m/2 \rceil} q^{\binom{m}{2}}.
\]
This generalizes the conclusion of Example~\ref{eg:geometric}.
\end{example}

\subsection{Generating functions}\label{ss:genfun-q}

We continue to assume that $H$ is a connected $q$-Hopf algebra that is cofree as a graded coalgebra, and $V=\calP(H)$. We also assume that $q\neq 0$.

Let
\[
v_q(t) := \sum_{n\geq 1} v_n \frac{t^n}{q^{\binom{n}{2}}}.
\]
All generating functions in this section will be of this form.

For each pair of nonnegative integers $k$ and $m$, let
\[
\pal_q(k,m) := \sum_{\substack{\alpha\,\vDash\, m\\ \ell(\alpha)=k}} \pal(\alpha)\,q^{\inv(\alpha)}.
\]
Then~\eqref{eq:S trace-q} may be rewritten as
\begin{equation}\label{eq:S trace-q-2}
 \tr\left(\apode\big\vert_{H_m}\right) = 
  	\sum_{k=0}^m (-1)^{k} \pal_q(k,m).
\end{equation}

We have the following $q$-generating functions for even and odd palindromes,
generalizing Proposition~\ref{p:genfun-pal}.
 
\begin{proposition}\label{p:genfun-pal-q}
\begin{align}
\label{eq:genfun-epal-q}
\sum_{k,m\geq 0} \pal_q(2k,m)\,s^k \frac{t^m}{q^{\binom{m}{2}}} & = \frac{1}{1-s v_{q^2}(t^2)},\\
\label{eq:genfun-opal-q}
\sum_{k,m\geq 0} \pal_q(2k+1,m)\,s^k \frac{t^m}{q^{\binom{m}{2}}} & = \frac{v_q(t)}{1-s v_{q^2}(t^2)}.
\end{align}
\end{proposition}
\begin{proof}
Consider a palindrome of even length $2k$ and weight $m>0$. Its multiweight $\alpha$
is a palindromic composition of $m$, necessarily of the form
\[
\alpha=(a,\beta,a),
\]
where $a$ is a positive integer and $\beta$ is a palindromic composition of $m-2a$. We have
\[
\inv(\alpha) = \inv(\beta) + a^2 + 2a(m-2a) 
\qand
\pal(\alpha) = v_a\,\pal(\beta).
\]
The former is equivalent to
\[
\inv(\alpha) - \binom{m}{2} = \inv(\beta) - \binom{m-2a}{2} - 2\,\binom{a}{2}.
\]
 Therefore,
\begin{multline*}
\frac{\pal_q(2k,m)}{q^{\binom{m}{2}}} = \sum_{\substack{\alpha\,\vDash\, m\\ \ell(\alpha)=k}} \pal(\alpha)\,q^{\inv(\alpha)-\binom{m}{2}}
= \sum_{a\geq 1} \sum_{\substack{\beta\,\vDash\, m-2a\\ \ell(\beta)=2k-2}} v_a\,\pal(\beta)\,
q^{\inv(\beta)-\binom{m-2a}{2}- 2\,\binom{a}{2}}\\
= \sum_{a\geq 1} \frac{v_a}{q^{2\,\binom{a}{2}}}\, \frac{\pal_q(2k-2,m-2a)}{q^{\binom{m-2a}{2}}}.
\end{multline*}
This recursion leads to~\eqref{eq:genfun-epal-q}. A similar argument establishes~\eqref{eq:genfun-opal-q}.
\end{proof}

We arrive at the following generalization of Corollary~\ref{c:genfun-trace}.

\begin{corollary}\label{c:genfun-trace-q}
\begin{equation}\label{eq:genfun-trace-q}
\sum_{m\geq 0} \tr\bigl(\apode\big\vert_{H_m}\bigr)\, \frac{t^m}{q^{\binom{m}{2}}} = \frac{1-v_q(t)}{1-v_{q^2}(t^2)}.
\end{equation}
\end{corollary}
\begin{proof}
This follows by subtracting~\eqref{eq:genfun-opal-q} from~\eqref{eq:genfun-epal-q}, letting $s=1$,
and employing~\eqref{eq:S trace-q-2}.
\end{proof}

\subsection{$q$-deformations}\label{ss:deformation}

The results in the Sections~\ref{ss:char-cofree-q} and~\ref{ss:genfun-q}
 apply only under the assumption of cofreeness. 
For $q$-Hopf algebras, this hypothesis is less restrictive than it may seem,
as we now argue.

Suppose our $q$-Hopf algebra is obtained by deforming the product
of an ordinary Hopf algebra, and leaving the unit and the coalgebra structure unchanged.
Thus, we have a family of products $\mu_q$ on a coalgebra $H$, turning it for each $q$
into a connected $q$-Hopf algebra, which we denote by $H(q)$.
Assume also that $\mu_q$ depends polynomially on $q$.
An example is $\cT_q(V)$, which is a deformation of $\cT(V)$.
More generally, the $q$-Hopf algebras constructed from Hopf monoids in species
by means of the functor $\cKc_{V,q}$, as in~\cite[\S 19.7]{AguMah:2010}, are all of this form.

In this situation, we may consider the $0$-Hopf algebra $H(0)$.
A result of Loday and Ronco~\cite[Thm. 2.6]{LodRon:2006} (see also~\cite[Thm. 2.13]{AguMah:2010}) guarantees that $H(0)$ is cofree as a graded coalgebra.
Since the coproduct has not been deformed, we have that our $q$-Hopf algebra $H(q)$ is cofree, for all $q$. 

In particular, the preceding results apply to  such Hopf algebra deformations.
By duality, they also apply in situations where the coproduct has been polynomially deformed
while the rest of the structure has been kept.

\subsection{$(-1)$-Hopf algebras}\label{ss:minusone}

The results of Section~\ref{s:main} relied on the PBW and CMM theorems for
graded connected Hopf algebras. While these results are not available for general
$q$-bialgebras, they are for $q=\pm 1$. In particular, the eigenvalues of the antipode
of a $(-1)$-Hopf algebra are still $\pm 1$, and the characteristic polynomials of the
Adams operators take the form~\eqref{eq:char poly}. 
(The multiplicities $\mul(k,m)$ are no longer given by~\eqref{eq:eul}.)


\bibliographystyle{plain}  
\bibliography{adams}

\end{document}